\documentclass[11pt,reqno]{amsart}
\usepackage{graphicx}
\usepackage[margin=3.5cm]{geometry}
\usepackage{mathtools}

\usepackage{amsfonts,latexsym,amssymb,amsmath,url,amsthm, color}
\usepackage[dvipsnames]{xcolor}

\usepackage[
    colorlinks=true,
    linkcolor=Violet,
    citecolor=Violet,
    urlcolor=Violet
]{hyperref}

\numberwithin{equation}{section}

\theoremstyle{plain}
\newtheorem{theorem}{Theorem}[section]
\newtheorem*{theorem*}{Theorem}
\newtheorem{lemma}[theorem]{Lemma}
\newtheorem{corollary}[theorem]{Corollary}
\newtheorem{proposition}[theorem]{Proposition}

\theoremstyle{definition}

\theoremstyle{remark}

\newtheorem*{ack}{Acknowledgements}
\newtheorem*{AI}{AI statement}

\title{Integer symplectic matrices in expanding boxes}

\newcommand{\Z}{\mathbb{Z}}
\newcommand{\R}{\mathbb{R}}

\newcommand{\Q}{\mathbb{Q}}

\newcommand{\x}{\mathbf{x}}
\newcommand{\z}{\mathbf{z}}
\newcommand{\y}{\mathbf{y}}
\newcommand{\w}{\mathbf{w}}

\newcommand{\cc}{\mathbf{c}}

\DeclareMathOperator{\SL}{SL}
\DeclareMathOperator{\GL}{GL}

\DeclareMathOperator{\rank}{rank}
\DeclareMathOperator{\Sp}{Sp}

\let\originalLambda\Lambda
\renewcommand{\Lambda}{\mathsf{\originalLambda}}

\let\originalGamma\Gamma
\renewcommand{\Gamma}{\mathsf{\originalGamma}}

\let\originalSigma\Sigma
\renewcommand{\Sigma}{\mathsf{\originalSigma}}

\let\originalXi\Xi
\renewcommand{\Xi}{\mathsf{\originalXi}}

\begin{document}

\author{Lena Wurzinger}
\address{Institute of Science and Technology Austria (ISTA)\\
Am Campus 1\\
3400 Klosterneuburg\\
Austria}
\email{lena.wurzinger@ist.ac.at}


\begin{abstract}
    We prove an essentially optimal uniform upper bound for the number of integer symplectic matrices inside a box of expanding side length and arbitrary centre. 
\end{abstract}
\maketitle

\setcounter{tocdepth}{1}
\tableofcontents

\thispagestyle{empty}

\section{Introduction}

Let $m \ge 1$ be a positive integer, let $\mathcal{M}_{2m}(\R)$ denote the set of $2m \times 2m$ real matrices, and let $\Sp_{m}(\R) \subseteq \mathcal{M}_{2m}(\R)$ denote the set of $2m \times 2m$ real symplectic matrices. Recall that $U \in \mathcal{M}_{2m}(\R)$ is  \emph{symplectic} if 
\begin{equation}
\label{eq:def_syplectic_matrix}
    U^t \Omega \, U = \Omega, \qquad 
    \text{ where }
    \qquad 
    \Omega = \begin{pmatrix}
        O_m & I_m \\
        -I_m & O_m
    \end{pmatrix},
\end{equation}
with $I_m$ denoting the $m\times m$ identity matrix and $O_m$ the $m\times m$ zero matrix. Similarly, we denote by $\mathcal{M}_{2m}(\Z)$ the set of all $2m \times 2m$ integer matrices and by $\Sp_m(\Z)$ the set of all $2m \times 2m$ integer symplectic matrices. 

Let $\mathbf{K}=(K_{i,j})_{i,j=1}^{2m} \in \mathcal{M}_{2m}(\R)$ and denote by 
\[
\mathfrak{B}_{\mathbf{K}}(T) = \prod_{i,j=1}^{2m} [K_{i,j}-T, K_{i,j}+T]
\]
 a box of side length $2T$. 
We prove a uniform upper bound for the number of symplectic matrices in $ \mathfrak{B}_{\mathbf{K}}(T) $, that is, for the counting function
\[
N_{m}( \mathfrak{B}_{\mathbf{K}}(T) ) = 
    \#
   \left( \Sp_{m}(\Z) \cap \mathfrak{B}_{\mathbf{K}}(T)
    \right).
\] 
Throughout, we use the convention that all implicit constants may depend on $m$.

\begin{theorem}\label{thm:upperbound}
    For $T\ge 1$, we have 
    \[
    N_{m}( \mathfrak{B}_{\mathbf{K}}(T) ) \ll_\varepsilon T^{m(m+1)+\varepsilon},
    \] 
    uniformly in $\mathbf{K} = (K_{i,j})_{i,j=1}^{2m} \in \mathcal{M}_{2m}(\R)$.
\end{theorem}

The upper bound in Theorem~\ref{thm:upperbound} is essentially sharp, as is  seen by considering boxes centred at the origin and applying a special case of work on counting lattice points of bounded norm on Lie groups due to Maucourant~\cite{maucourant}; details are provided in Section~\ref{section:asymptotic}.

\begin{theorem}[Maucourant]
   \label{theorem:asymptotic}
    There is a constant $a_m>0$ such that 
    \[
     N_{m}( \mathfrak{B}_{\mathbf{0}}(T) ) \sim a_m T^{m(m+1)}  
    \]
    as $T \to \infty$.
\end{theorem}

For a general centre $\mathbf{K}$, however, the box $\mathfrak B_{\mathbf K}(T)$ is no longer a norm ball covered by such results; achieving uniformity in the location of the centre of the box constitutes the main feature of our work. 
While results of Gorodnik--Nevo~\cite{gorodnik_nevo} allow for lattice point counting on considerably more general families of domains, they would not immediately yield Theorem~\ref{thm:upperbound}. Firstly, their results require certain well-roundedness conditions,
which amount to controlling how the Haar measure of the domain changes under small left and right perturbations in the group. In our case, one would consider, for $U \in \mathfrak{B}_{\mathbf{K}}(T)$, perturbations 
\[
U\mapsto h_1Uh_2
\]
where $h_1,h_2$ are close (say, in an $\varepsilon$-neighbourhood) to the identity. The resulting change of the entries of $U$ is of order $\ll \varepsilon |U|$. Now, when $|\mathbf{K}|$ is large compared to $T$, any matrices in $\mathfrak{B}_{\mathbf{K}}(T)\cap \Sp_m(\R)$ will be of order $|U| \asymp |\mathbf{K}|$. Therefore the control on the perturbation cannot be easily made uniform in terms of $\mathbf{K}$.
Secondly, one would need to control the Haar measure of $\mathfrak{B}_{\mathbf{K}} (T) \cap \Sp_m(\R)$ uniformly; unlike for boxes centred at the origin, these domains are not norm balls, and the required volume estimates do not appear to be directly accessible from the standard Cartan-decomposition arguments.

When $m=1$, we have 
\[
    \Sp_1(\mathbb{Z}) =
    \operatorname{SL}_2(\mathbb{Z}) = \{M \in \mathcal{M}_2(\Z):~\det(M)=1 \}.
\]
Krieg~\cite{krieg_maximum_norm} proved that the number of matrices 
of maximum norm $|M| = N$ in $\SL_{2}(\Z)$ is $32 \varphi (N)$ for $N \ge 2$, where $\varphi$ denotes Euler's totient function. This yields 
\[
     N_{1}( \mathfrak{B}_{\mathbf{0}}(T) ) = 32 \sum_{2\le N \le T} \varphi (N) + O(1) = a_1T^2+O(T\log T)
\]
as $T \to \infty$. 
The closely related question of counting $M \in \SL_2(\Z)$ of bounded Frobenius norm $\|M\|_F$ is equivalent to the classical hyperbolic lattice point problem. Selberg's result \cite[\S12]{iwaniec_spectral} gives
\begin{equation}\label{eq:selberg}
    \#\{M \in\SL_2(\Z):
           \|M\|_F\le T\}
= cT^2+O(T^{4/3})
\end{equation}
where $c>0$ is a positive constant. 
Quantitative asymptotic formulas for the same question on $\SL_n(\Z)$, $n\ge 3$, can be found in work of Blomer--Lutsko
\cite{blomerlutsko}, improving upon earlier error terms due to Gorodnik--Nevo--Yehoshua \cite{gorodnik_nevo_yehoshua} and Duke--Rudnick--Sarnak~\cite{duke_rudnick_sarnak}.

Shparlinski \cite{shparlinski_counting_matrices} showed that the number of integer matrices of fixed determinant with entries contained inside an expanding box is bounded by 
 \begin{equation}\label{ eq:shparlinski_detbound}
   \# \{ 
    M \in 
    \mathcal{M}_n(\Z) \cap \mathfrak{B}_{\mathbf{K}}(T):~\det(M)=d
    \} \ll T^{n^2 - n} \log(T)
 \end{equation}
 for any integer $d$, uniformly in $d$ and $\mathbf{K}$;  thus Theorem~\ref{thm:upperbound} follows from \eqref{ eq:shparlinski_detbound} in the case $m=1$. For boxes centred at the origin, asymptotic formulas due to  
  Duke--Rudnick--Sarnak~\cite{duke_rudnick_sarnak} (for $d\neq 0$) and Katznelson~\cite{katznelson_singular_matrices} (for $d = 0$) yield
 \[
 \# \{ 
    M \in 
    \mathcal{M}_n(\Z) \cap \mathfrak{B}_{\mathbf{0}}(T):~\det(M)=d
    \} 
    \asymp \begin{cases}
        T^{n^2-n} & \text{if } d\neq 0, \\
        T^{n^2-n} \log(T)
        & \text{if } d=0,
    \end{cases}
    \]
    where $d$ is fixed.
 Recent work of Afifurrahman~\cite{afifurrahman_uniform_formula} concerning $2\times2$ integer matrices, subsequently  refined by Dhanda--Haynes--Prasala~\cite{dhanda_haynes_prasala} and then by Chapman-Mugdal~\cite{chapman_mudgal_matrices, chapman_mudgal_quadratic}, proves that the asymptotic formulas hold uniformly for $d\ne 0$ in ranges growing with $T$.

More broadly, various counting problems involving matrices with
prescribed algebraic or arithmetic properties have received considerable attention.
In a direction closely related to \eqref{eq:selberg},
Friedlander--Iwaniec \cite{friedlander_iwaniec}
studied the hyperbolic prime number theorem,
concerning prime values of the squared Frobenius norm
$\|M\|_F^2$ for $M\in\SL_2(\Z)$; they give upper and lower bounds of the same order, subject to a standard conjecture
concerning primes in arithmetic progressions.
Very recently, Shparlinski--Xiao
\cite{shparlinskixiao2026} and Sedunova
\cite{sedunovahyperbolic} obtained unconditional
asymptotic formulas with prime values replaced by
squarefree values, under maximum norm and
Frobenius norm constraints, respectively. 
Further recent work on matrix-counting problems includes uniform bounds for
integer matrices with a prescribed characteristic
polynomial due to Habegger--Ostafe--Shparlinski
\cite{habegger_ostafe_shparlinski}, estimates
for rational matrices of prescribed rank, determinant
or characteristic polynomial due to
Afifurrahman et al.\
\cite{afifurrahman_kuperberg_ostafe_shparlinski}
and bounds for integer matrix
factorisations due to Afifurrahman  \cite{afifurrahman_matrix_products} and Browning
et al.~\cite{browningetal_divisor}.

 \subsection*{Proof strategy} 
 We begin the proof of Theorem~\ref{thm:upperbound}
by associating to each matrix $U \in \mathfrak{B}_{\mathbf{K}}(T) \cap \Sp_m(\Z)$ the primitive isotropic lattice $\Lambda_U$ that is generated by the first $m$ columns of $U$. Here, \emph{isotropic} means that 
\[
    \x^t \Omega \, \y = 0 \text{ for all }\x,\y \in \Lambda_U,
\]
and \emph{primitive} means that 
\[
    \mathrm{span}_\R (\Lambda_U) \cap \Z^{2m} = \Lambda_U.
\]

The ambiguity in completing the first $m$ columns to a symplectic integer matrix in $\mathfrak{B}_{\mathbf{K}}(T)$ is controlled by the sufficiently short (of size $\ll T$) vectors of $\Lambda_U$. Indeed, fixing the first $m$ columns $\x_1,\ldots,\x_m$, consider two extensions $(\x_1~\cdots~\x_m~\y_1~\cdots~\y_m)$ and $(\x_1~\cdots~\x_m~\widetilde{\y}_1~\cdots~\widetilde{\y}_m)$ to admissible matrices in $\mathfrak{B}_{\mathbf{K}}(T) \cap \Sp_m(\Z)$. By the box constraint, we have $|\y_i- \widetilde{\y}_i| \ll T$, while we will see that 
\[
\y_i - \widetilde{\y}_i \in \Lambda_U, 
\]
for $i=1,\ldots,m$. 
For instance, if $\Lambda_U$ contains no sufficiently short nonzero vectors, this implies that the first $m$ columns have at most one completion.

We exploit two complementary phenomena: short vectors can be handled by techniques from the geometry of numbers, while the absence of short vectors imposes more rigidity. 
Therefore it is natural to decompose the set 
$ \mathfrak{B}_{\mathbf{K}}(T) \cap \Sp_m(\Z) $ according to the rank $q$  of the lattice $\Sigma_U$, which is the primitive lattice given by the intersection of the $\R$-span of all short vectors in $\Lambda_U$ with $\Lambda_U$ itself.
Let us explain what happens in the two extremal cases $q=0$ and $q=m$.

For $q=0$, there are \emph{no short directions}.
As discussed above, for each choice of $\x_1,\ldots,\x_m$, there is at most one possible extension. It therefore suffices to bound the number of choices for isotropic tuples $(\x_1,\ldots,\x_m)$. By Proposition~\ref{prop:translated_isotropic_frames}, this number is 
\[
    \ll_\varepsilon T^{m(m+1)+\varepsilon}.
\]

By contrast, for $q = m$, \emph{all directions are short} and we have $\Lambda_U = \Sigma_U$.  We first decompose the successive minima $\sigma_1 \le \cdots \le \sigma_m \ll T$ of
$\Sigma_U$  into $\ll_\varepsilon T^\varepsilon$ dyadic ranges $\theta_i \le \sigma_i \le 2\theta_i$; for each range, we have by Corollary~\ref{cor:succminima} that   
\[
    \# \{ 
    \Sigma \text{ prim.~isotropic rank-$m$ lattice in $ \Z^{2m}$}:~ \theta_i \le \sigma_i \le 2\theta_i
    \} \ll_\varepsilon T^{\varepsilon}\prod_{i=1}^m \theta_i^{2(m-i+1)}.
\]
We obtain this upper bound  by encoding our lattices by primitive isotropic flags and applying bounds due to Franke--Manin--Tschinkel \cite{franke_manin_tschinkel} for rational points of bounded height on symplectic flag varieties. For a fixed lattice $\Sigma$ with successive minima in a range as above, by Lemma~\ref{lem:boundingpersigmaw} the number of associated symplectic matrices $U$ is at most 
\[
    \# 
    \{U  \in \mathfrak{B}_{\mathbf{K}}(T) \cap \Sp_m(\Z):~\Sigma_U = \Sigma 
    \} \ll_\varepsilon T^{m(m+1)+\varepsilon}
    \prod_{i=1}^m \theta_i^{-2(m-i+1)}.
\]
The dependence
on the successive minima arising in these two estimates thus cancels exactly.

The proof in the intermediate range $1 \le q \le m-1$ may be viewed as an interpolation between the two extremal cases: the lattice
$\Sigma_U$ records the short directions, while the factor lattice
$\Lambda_U/\Sigma_U$ has no short nonzero vectors. In this way, the proof separates the flexible short directions, which are treated similarly to the case $q=m$, from the rigid long directions, which impose constraints analogous to those in the case $q=0$.

\begin{ack}
    The author is very grateful to her advisor Tim Browning for his support and feedback throughout this project and for communicating the question addressed in Theorem~\ref{thm:upperbound}, which was posed by Igor Shparlinski. She would also like to thank Igor Shparlinski and Alina Ostafe for very interesting discussions, and Fabian Gundlach and Igor Shparlinski for feedback on an earlier version of this draft. Moreover, she would like to thank Nikita P.~Kalinin for his assistance with AI tools.
\end{ack}

\begin{AI}
    Generative AI tools were used to create the proof outline of Theorem~\ref{thm:upperbound}; specifically, GPT-5.6 Sol and Claude Opus 5 were employed. GPT-6 Astra was used for literature research and proofreading. All arguments were independently checked by the author, who takes full responsibility for the contents of the paper. The final text was written by the author.
\end{AI}

\section{Preliminaries} 

In this section, we collect some results and notation that we will need for our proof of Theorem~\ref{thm:upperbound} and for our discussion of Theorem~\ref{theorem:asymptotic}.
We write 
\[
|M|=\max_{i,j}|M_{i,j}|, \quad \text{where }M = (M_{i,j})_{\substack{1 \le i \le n \\ 1\le j \le k}} \in \R^{n\times k},
\]
 and we denote the Euclidean norm of a vector $\x \in \R^n$ by $\| \x \|_2$.

\subsection{Geometry of numbers} 
We recall some facts from the geometry of numbers that will be used throughout. For further background, we refer to \cite{cassels_geometry_numbers}.

Let $n \ge 1$ be a positive integer and let $\Lambda \subseteq \R^{n}$ be a lattice of rank $r \le n$. 
The \emph{$k$-th successive minimum} $\lambda_k$, for $k \le r$, is the smallest radius of a closed Euclidean ball centred at the origin that contains $k$ linearly independent points of $\Lambda$. Its determinant is given by 
\[
\det(\Lambda) = \sqrt{\det(B_\Lambda^t B_\Lambda)},
\]
where the columns of $B_\Lambda$ form a basis of $\Lambda$. By Minkowski's second theorem, we have 
\begin{equation}\label{eq:minkowski}
    \det(\Lambda) \ll \lambda_1 \cdots \lambda_r \ll \det(\Lambda),
\end{equation}
with the implicit constants depending only on $n$.

For any lattice, one may choose a basis adapted to its successive minima, with control on the coefficients of lattice vectors in this basis. The following statement follows from  \cite[Lemma~5]{davenport_cubic_forms}; we use a formulation similar to the one given in \cite[Lemma~1]{browning_heath_brown}.

\begin{lemma}\label{lem:basis_successive_min_size}
     Let $\Lambda \subseteq \R^n$ be a lattice of rank $r$ with successive minima $\lambda_1\le \cdots \le \lambda_r$. Then there is a basis $\z_1, \ldots, \z_r$ of $\Lambda$ such that we have 
     \[
        |\z_i| \ll \lambda_i,
     \]
     and with the property that whenever
\[
    \x=\sum_{i=1}^r a_i\z_i\in\Lambda,
\]
we have
$
    |a_i|\ll |\x| / \lambda_i,
$
for $i=1,\ldots,r$.     All implicit constants only depend on $n$.
\end{lemma}

For a lattice $\Lambda$ as in Lemma~\ref{lem:basis_successive_min_size} and $\x \in \mathrm{span}_{\R}(\Lambda)$, say $\x = \sum_{i=1}^r \mu_i \z_i$, we observe that there is $\y = \sum_{i=1}^r \lfloor \mu_i \rfloor \z_i \in \Lambda$
with 
\[
    |\x - \y| \ll \sum_{i=1}^r |\z_i| \ll \lambda_r;
\]
this implies the following lemma.

\begin{lemma}\label{lem:close_to_vector}
     Let $\Lambda \subseteq \R^n$ be a lattice of rank $r$ with successive minima $\lambda_1\le \cdots \le \lambda_r$. Assume that $\lambda_r \le  T$. Then, for every $\x \in \mathrm{span}_{\R}(\Lambda)$, there exists $\y \in \Lambda$ such that 
     \[
     | \x - \y | \le C_n T.
     \]
     Here $C_n$ is a positive constant only depending on $n$.
\end{lemma}

For a bounded closed set $\mathcal{S} \subseteq \R^n$, we define the \emph{diameter} of $\mathcal{S}$ to be the supremum of the Euclidean distance between two points in $\mathcal{S}$. For sets $\mathcal{S}$ of diameter $\le T$, an upper bound for $\# (\mathcal{S} \cap \Lambda)$ easily follows from  Lemma~\ref{lem:basis_successive_min_size}. Assuming without loss of generality that the intersection is nonempty, let $\x_0 \in \mathcal{S} \cap \Lambda$, and write 
\[
    \x_0 = \sum_{i=1}^r a_i \z_i
\]
for some $a_i \in \Z$, using the basis supplied by Lemma~\ref{lem:basis_successive_min_size}. If $\x \in \mathcal{S} \cap \Lambda$, then $\x$ may be written as 
\[
   \x = \sum_{i=1}^r b_i \z_i.
\]
Since $|\x - \x_0| \ll T$, it follows by Lemma~\ref{lem:basis_successive_min_size} that 
\[
    |a_i - b_i| \ll T/\lambda_i
\]
for $i=1,\ldots,r$, that is, there are $\ll 1 + T/\lambda_i$ different choices for $b_i$. This yields the following standard bound, which appears in many variations in the literature (see e.g.~\cite[Lemma~2]{schmidt1968}). 

\begin{lemma}\label{lem:latticepointboundbox}
    Let $\Lambda \subseteq \R^n$ be a lattice of rank $r$ with successive minima $\lambda_1\le \cdots \le \lambda_r$. 
For a set $\mathcal{S}$ of diameter $\le T$, the number of lattice points in $\mathcal{S}$ can be bounded as 
        \[
        \#
        \left( \mathcal{S} \cap \Lambda
        \right) \ll \prod_{j=1}^r 
        \left( 
        1 + \frac{T}{\lambda_j}
        \right).
        \]
        The implicit constant depends only on $n$.
\end{lemma}

For lattices $\Lambda \subseteq \Gamma \subseteq \R^n$, we call $\Lambda$ \emph{primitive in $\Gamma$} if 
\[
\mathrm{span}_{\R}(\Lambda) \cap \Gamma = \Lambda.
\]
We call a lattice \emph{primitive} if it is primitive in $\Z^n$. 
For a primitive sublattice $\Lambda \subseteq \Gamma$,  denote by $\pi_{\Lambda}: \R^n \to \Lambda^\perp$ the projection onto the orthogonal complement of $\Lambda$.
Then the \emph{factor lattice} $\Gamma / \Lambda$ is defined as (see e.g.~\cite[{\S 2}]{schmidt1968}) 
\begin{equation}\label{eq:factor_lattice}
      \Gamma / \Lambda = \pi_{\Lambda}(\Gamma).
\end{equation}
Since $\Lambda$ is primitive in $\Gamma$, this definition of $\Gamma / \Lambda$ indeed yields a lattice inside the vector space $\Lambda^\perp$ with $\mathrm{rank}(\Gamma / \Lambda)=\mathrm{rank}(\Gamma)-\mathrm{rank}(\Lambda)$. 

We have the identity 
\begin{equation}\label{eq:factorlatticeid}
       \det(\Lambda) \det(\Gamma / \Lambda) = \det(\Gamma).
\end{equation}
The following result explains that primitive sublattices remain primitive after passing to a factor lattice. 

\begin{lemma}\label{lem:primitivityfactsetc}
    Let $\Sigma \subseteq \Lambda \subseteq \Gamma$ be lattices. Assume that $\Sigma$ is primitive in $\Lambda$ and that $\Lambda$ is primitive in $\Gamma$. Then $\Sigma$ is primitive in $\Gamma$ and $\Lambda / \Sigma$ is primitive in $\Gamma / \Sigma$. 
\end{lemma}

\subsection{Symplectic geometry} We recall a few standard definitions from symplectic geometry. With $\Omega$ as in \eqref{eq:def_syplectic_matrix}, we define the \emph{standard symplectic form} as 
\begin{equation}\label{def:omega}
    \omega: \R^{2m} \times \R^{2m} \to \R, \quad (\x,\y) \mapsto \x^t \Omega \, \y.
\end{equation}
For any set $\mathcal{S} \subseteq \R^{2m}$, we define the \emph{symplectic complement} of $\mathcal{S}$ as 
\begin{equation*}
        \mathcal{S}^{\perp_\omega} = 
    \{ 
    \x \in \R^{2m}:~\omega(\x,\y) = 0 \text{ for all }\y \in \mathcal{S}
    \}.
\end{equation*}
We call a subspace $V \subseteq \R^{2m}$ \emph{isotropic} if
\[
   V \subseteq  V^{\perp_\omega},
\]
and we call $V$ \emph{Lagrangian} if
\[
    V=V^{\perp_\omega}.
\]
By linear algebra, 
\begin{equation}\label{eq:dimorthogonalcomplement}
        \dim(V^{\perp_\omega}) = 2m - \dim(V),
\end{equation}
so an isotropic subspace of $\R^{2m}$ is Lagrangian if and only if it is $m$-dimensional.
We call a lattice $\Lambda \subseteq \R^{2m}$ \emph{isotropic},
respectively \emph{Lagrangian}, if
$\mathrm{span}_{\R}(\Lambda)$ is isotropic, respectively Lagrangian.

For a primitive isotropic lattice $\Sigma$ of rank $q \le m$, we set 
\begin{equation}\label{eq:def_gamma_sigma}
     \Gamma_\Sigma = \Sigma^{\perp_\omega} \cap \Z^{2m}.
\end{equation}
Clearly $ \Gamma_\Sigma $ is a primitive lattice, and since $\Sigma^{\perp_\omega}$
is a rational subspace, we get
\[
\rank(\Gamma_\Sigma) = \dim(\Sigma^{\perp_\omega}) = 2m-\rank(\Sigma) =2m-q. \] Since $\Sigma$ is isotropic, it satisfies the inclusion $\Sigma \subseteq \Gamma_\Sigma$. Now, because 
\[
    \Sigma \subseteq 
    \mathrm{span}_{\R}(\Sigma) \cap \Gamma_\Sigma  \subseteq 
    \mathrm{span}_{\R}(\Sigma) \cap \Z^{2m} = \Sigma,
\]
the lattice $\Sigma$ is primitive in $\Gamma_\Sigma$. 
As described in \eqref{eq:factor_lattice}, the projection $\pi_{\Sigma}(\Gamma_\Sigma)$
yields the factor lattice $\Gamma_\Sigma / \Sigma$ and 
\[
\rank (\Gamma_\Sigma / \Sigma) = \rank(\Gamma_\Sigma) - \rank(\Sigma) = 2(m-q).
\]
 
\begin{lemma}
    The form $\omega$ descends to a form 
\[
    \overline{\omega}: \Gamma_\Sigma / \Sigma \times \Gamma_\Sigma / \Sigma \to \mathbb{R}, \qquad  \overline{\omega}(\pi_{\Sigma}(\x_1), \pi_{\Sigma}(\x_2)) = \omega (\x_1,\x_2).
\]
 The resulting form  $\overline{\omega}$ is \emph{non-degenerate}, that is, for $\x \in \Gamma_\Sigma$ the relation
\[
    \overline{\omega}(\pi_{\Sigma}(\x),\pi_{\Sigma}(\y)) = 0 \quad \text{for all } \y \in \Gamma_\Sigma
\]
implies that $\pi_{\Sigma}(\x)=\mathbf{0}$.
\end{lemma}
\begin{proof}
     For any $\boldsymbol{\sigma}_1, \boldsymbol{\sigma}_2 \in \Sigma$ and any $\x_1, \x_2 \in \Gamma_\Sigma$, the bilinearity of $\omega$ gives
\[
    \omega(\x_1+\boldsymbol{\sigma}_1, \x_2+\boldsymbol{\sigma}_2)
    =
    \omega(\x_1 , \x_2)
    +
    \omega(\x_1, \boldsymbol{\sigma}_2) +
    \omega(\boldsymbol{\sigma}_1, \x_2)
    + 
    \omega(\boldsymbol{\sigma}_1, \boldsymbol{\sigma}_2) = \omega(\x_1, \x_2),
\]
so $\overline{\omega}$ is well-defined.  
It remains to prove that $\overline{\omega}$ is non-degenerate. The relation
\[
    \overline{\omega}(\pi_{\Sigma}(\x),\pi_{\Sigma}(\y)) = \omega (\x,\y) = 0 \quad \text{for all } \y \in \Gamma_\Sigma
\]
implies that $\x \in \Gamma_\Sigma^{\perp_\omega}$. 
By the definition of $\Gamma_\Sigma$ and \eqref{eq:dimorthogonalcomplement}, we have
$
    \mathrm{span}_{\R}(\Gamma_\Sigma) = \Sigma^{\perp_\omega}
$
and therefore
\[
\Gamma_\Sigma^{\perp_\omega} = \big( 
\Sigma^{\perp_\omega}
\big)^{\perp_{\omega}} = \mathrm{span}_{\R}(\Sigma);
\]
 thus $\pi_{\Sigma}(\x) = \mathbf{0}$.
\end{proof}

\begin{lemma}\label{lem:lifting}
    Let $\Xi \subseteq \Gamma_\Sigma / \Sigma$ be a sublattice that is primitive in $\Gamma_\Sigma / \Sigma$, and set 
    $
    \Xi^{\uparrow} = \pi_{\Sigma}^{-1}\left(\Xi
    \right) \cap \Gamma_\Sigma.
    $
    Then 
    \begin{enumerate}
        \item $ \Xi^{\uparrow}  $ is a primitive lattice of rank $\rank(\Sigma)+\rank(\Xi)$ in $\Z^{2m}$, 
        \item $\det(  \Xi^{\uparrow} ) = \det( \Xi ) \det( \Sigma )$, and
        \item if $\Xi$ is isotropic with respect to $\overline{\omega}$, then $ \Xi^{\uparrow} $ is isotropic with respect to $\omega$.
    \end{enumerate}
\end{lemma}

\begin{proof}
In order to prove (1), it suffices by Lemma~\ref{lem:primitivityfactsetc} to prove that $\Xi^\uparrow$ is primitive in the primitive lattice $\Gamma_\Sigma$. Assume that 
\[
\x \in \mathrm{span}_{\R}(  \Xi^{\uparrow} ) \cap \Gamma_\Sigma .
\]
Then, by the primitivity of $\Xi$, we have   
\[
    \pi_{\Sigma}(\x) \in \mathrm{span}_{\R}(\Xi) \cap (\Gamma_\Sigma / \Sigma) = \Xi,
\]
and thus $\x \in \pi_{\Sigma}^{-1}(\Xi) \cap \Gamma_\Sigma = \Xi^\uparrow$. This proves that 
\[
    \Xi^\uparrow = \mathrm{span}_{\R}(  \Xi^{\uparrow} ) \cap \Gamma_\Sigma ,
\]
so $ \Xi^\uparrow$ is primitive. Moreover, since $\Sigma\subseteq \Xi^\uparrow$ and
$\pi_\Sigma(\Xi^\uparrow)=\Xi$, we have
$
\Xi^\uparrow/\Sigma = \Xi.
$
Hence
\[
\rank(\Xi^\uparrow)
=
\rank(\Sigma)+\rank(\Xi).
\]
 Statement (2) follows directly from \eqref{eq:factorlatticeid}. 
Statement (3) follows from the fact that $\Xi= \Xi^\uparrow / \Sigma$ and the definition  $\omega(\x,\y) = \overline{\omega}(\pi_{\Sigma}(\x),\pi_{\Sigma}(\y)) $.
\end{proof}

\begin{lemma}
    Every nonzero $\x \in \Gamma_\Sigma / \Sigma$ satisfies
   $
    |\x| \ge (\sqrt{2m}\det(\Sigma))^{-1}.
    $
\end{lemma}

\begin{proof}
    Let $\mathbf{0} \neq \x \in \Gamma_\Sigma / \Sigma$. Since $\x$ lies in the orthogonal complement of $\Sigma$, the lattice $\Gamma$ generated by the union of $\Sigma$ and $\y \in \pi_{\Sigma}^{-1}(\x) \cap \Z^{2m}$ has determinant 
    \[
        \det(\Gamma) = \det(\Sigma) \| \x \|_2.
    \]
    Since $\Gamma$ is an integer lattice, $\det(\Gamma) \ge 1$. Thus 
    $
        \sqrt{2m}|\x| \ge \|\x\|_2 \ge \det(\Sigma)^{-1}. $
\end{proof}

\subsection{Rational points on generalised flag varieties}
\label{section:rationalpoints_fmt}

We recall a consequence of a result due to Franke--Manin--Tschinkel on the distribution of rational points on generalised flag varieties, see \cite{franke_manin_tschinkel,franke_manin_tschinkel_erratum}.

Let $G$ be a semisimple linear algebraic group over $\Q$, let
$P\subseteq G$ be a $\Q$-rational parabolic subgroup and set $X = G/P$.
We denote by $H$ an anticanonical height on the set of rational points $X(\mathbb{Q})$; for any other anticanonical height $H_2$, we have $H \asymp H_2$. Franke--Manin--Tschinkel \cite[\S5]{franke_manin_tschinkel} prove that the associated height zeta function converges
for $\Re(s)>1$, which implies the upper bound
\begin{equation}\label{eq:rat_pts_upperbound_fmt}
\#\big\{
x\in X(\Q):
H(x)\le B
\big\}
\ll_{\varepsilon, X, H} B^{1+\varepsilon}.
\end{equation}
 The implicit constant may depend on $X$ and $H$ as well as on $\varepsilon$.

We will apply \eqref{eq:rat_pts_upperbound_fmt} in Section~\ref{sec:counting_isotropic_flags} to suitable symplectic flag varieties. There the relevant anticanonical height will be expressed explicitly in terms of determinants of primitive isotropic lattices.

We remark that Franke--Manin--Tschinkel actually prove an asymptotic formula for the number of points of bounded height. In several special cases, such asymptotics have also been obtained using methods from the geometry of numbers. Schmidt~\cite{schmidt1968} obtained asymptotic
formulas for rational points of bounded height on
Grassmannians over $\Q$, while Thunder~\cite{thunder_flags} 
obtained asymptotics for flag
varieties over number fields. More recently, using methods from homogeneous
dynamics,
Horesh--Karasik~\cite{horesh_karasik_flags} proved asymptotic formulas with respect to geometrical properties for flags over $\Q$. 

\subsection{Counting lattice points in Lie groups}\label{section:prelim_lie}
We present a result of Maucourant~\cite{maucourant}. The material in this subsection is used only in Section~\ref{section:asymptotic}, where we deduce Theorem~\ref{theorem:asymptotic}. Maucourant's result applies to connected, noncompact, semisimple real Lie groups
$G$ with finite centre containing an irreducible lattice $\Gamma$.
We will specialise to $\Sp_m(\R)$ in Section~\ref{section:asymptotic}.
 
    Let $\mathfrak{g}$ denote the Lie algebra of $G$, let $[\cdot,\cdot]: \mathfrak{g} \times \mathfrak{g} \to \mathfrak{g}$ denote the Lie bracket, and denote by  $\exp: \mathfrak{g} \to G$ the exponential map. Let $\theta: \mathfrak{g} \to \mathfrak{g}$ be a \emph{Cartan involution}, that is, a Lie algebra automorphism such that $\theta \circ \theta = \mathrm{id}$ and such that the bilinear form
\[
B_\theta(X,Y):=-B(X,\theta(Y))
\]
is positive definite, where $B$ denotes the Killing form of $\mathfrak{g}$.
Then $\theta$ gives rise to a \emph{Cartan decomposition}
$
\mathfrak{g}=\mathfrak{k}\oplus\mathfrak{p},
$
where 
\[
    \mathfrak{k} = 
    \{
    X \in \mathfrak{g}:~\theta(X)=X 
    \}, \qquad 
    \mathfrak{p} = 
    \{
    X \in \mathfrak{g}:~\theta(X)=-X 
    \}.
\]

Let $\mathfrak{a}\subseteq\mathfrak{p}$ be a \emph{maximal abelian subspace}, that is,
$
    [X,Y] = 0 \text{ for all $X,Y \in \mathfrak{a}$}
$
and $\mathfrak{a}$ is inclusion-maximal among subspaces of $\mathfrak{p}$ with this property. 
For a linear functional
$
 \alpha:\mathfrak a\to\mathbb {R} 
$
in the dual space $\mathfrak{a}^\star$,
we define 
\[
    \mathfrak{g}_\alpha = 
    \{X \in \mathfrak{g}: [H,X]=\alpha(H)X  \text{ for all } H \in \mathfrak{a}
    \}.
\]
A nonzero functional $\alpha$ is called a \emph{restricted root} of \emph{multiplicity} $m_\alpha$ if $m_\alpha := \dim (\mathfrak{g}_\alpha)\ge 1$. The restricted roots occur in pairs $\{\alpha,-\alpha\}$. A choice
$\mathcal R^+$ of \emph{positive restricted roots} may be obtained by
choosing a regular element $H_0\in\mathfrak a$, that is, an element
such that
$
\alpha(H_0)\neq0
$
for every restricted root $\alpha$, and setting
$
\mathcal R^+
=
\{\alpha:\alpha(H_0)>0\}.
$
We then set
\begin{equation*}
    \beta
    =
    \sum_{\alpha\in\mathcal R^+}m_\alpha\alpha.
\end{equation*}

Let
    $
    \rho: G \to \GL (V)
    $
    be a faithful finite-dimensional real representation.
Following Maucourant \cite[\S\S~2.1 and~2.3]{maucourant}, let
$\chi_1,\ldots,\chi_n \in \mathfrak{a}^\star$ denote the weights of
the representation $\rho$. We choose a basis
of $V$ such that
\[
\rho(\exp(H))
=
\operatorname{diag}\bigl(
e^{\chi_1(H)},\ldots,e^{\chi_n(H)}
\bigr)
\]
for every $H \in \mathfrak{a}$ and define the \emph{weight polytope} $\mathcal{C}_\rho$ as the convex hull of these weights, i.e.~as
$
    \mathcal{C}_\rho = \mathrm{conv}(\chi_1,\ldots,\chi_n).
$
Under these hypotheses (see \cite[Lemma~2.1]{maucourant}), $\mathbf{0} \in \mathfrak{a}^\star$ lies in the
interior of $ \mathcal{C}_\rho $. Let $d>0$ be the smallest positive real number such that
\begin{equation}\label{eq:def_d}
    d^{-1}\beta\in\mathcal C_\rho.
\end{equation}
By the minimality of $d$, $d^{-1}\beta$ lies on the boundary $\partial  \mathcal{C}_\rho$. 
Let $F_\beta$ be the smallest face of 
$\mathcal C_\rho$ containing $d^{-1}\beta$, and set
\begin{equation}\label{eq:def_e}
    e=\operatorname{codim}_{\mathfrak{a}^\star}(F_\beta)-1.
\end{equation}

Maucourant's lattice-point counting result now takes the following form.

\begin{theorem}[{\cite[Corollary~1.1]{maucourant}}] \label{thm:mau}
    Let $\| \cdot \|$ be any norm on the set of endomorphisms $\mathrm{End}(V)$. Then there is a positive constant $a>0$ such that 
    \[
    \# \{ 
    \gamma \in \Gamma:~\|\rho(\gamma)\| \le T 
    \} \sim a T^d (\log T)^e
    \]
    as $T \to \infty$, 
    where $d$ is defined in \eqref{eq:def_d} and $e$ is defined in \eqref{eq:def_e}.
\end{theorem}

\section{Counting isotropic configurations}\label{sec:counting_isotropic_flags}

We establish upper bounds for the number of certain isotropic configurations that will be essential for the proof of Theorem~\ref{thm:upperbound}. 

In Section~\ref{section:fmt_lattices}, we count isotropic lattices with successive minima contained in prescribed dyadic ranges. More concretely, we count pairs $(\Sigma, \Xi)$, where $\Sigma$ is a primitive isotropic lattice in $\Z^{2m}$ and $\Xi$ is a primitive isotropic lattice contained in $\Gamma_\Sigma/\Sigma$ (with $\Gamma_\Sigma$ defined in \eqref{eq:def_gamma_sigma}). As explained in the proof-strategy discussion in the introduction, we encode the short directions of the lattice generated by the first $m$ columns of a symplectic matrix $U$ by a primitive isotropic lattice $\Sigma_U$. After passing to the factor lattice $\Gamma_{\Sigma_U}/\Sigma_U$, we will encounter auxiliary primitive isotropic lattices $\Xi$ generated by differences of projected columns.

In
Section~\ref{section:tuples}, we bound the number of isotropic tuples of integer vectors $(\x_1,\ldots,\x_r)$ with $\x_i \in \mathcal{D}_i$ as a function of the diameter of the sets $\mathcal{D}_i$, uniformly in their locations. Ultimately, this will allow us to bound the number of possible choices for the first $m$ columns of a symplectic matrix in $\mathfrak{B}_{\mathbf{K}}(T)$, and, later in the argument, to control the possible completions to a symplectic matrix.

\subsection{Counting isotropic lattices} \label{section:fmt_lattices}
In this section, we apply the rational-point estimate due to  Franke--Manin--Tschinkel \cite{franke_manin_tschinkel} from Section~\ref{section:rationalpoints_fmt} to symplectic flag varieties. 
A special case of our result (obtained by taking $s=0$ in Proposition~\ref{prop:countingflags} below) is the following.
\begin{corollary}\label{cor:succminima}
        Let $0\le q\le m$ be a nonnegative integer and let $1 \le \theta_1 \le \cdots \le \theta_q \le T$  be positive real numbers.
    Then the number of primitive isotropic lattices $\Sigma \subseteq \Z^{2m}$ of rank $q$ with successive minima $\sigma_1 \le \cdots \le \sigma_q $ satisfying $ \theta_i \le \sigma_i \le 2 \theta_i$ for $1\le i \le q$ is 
    \[
     \ll_\varepsilon T^\varepsilon \prod_{i=1}^{q} \theta_i^{2(m-i+1)} .
    \]
    The implicit constant only depends on $\varepsilon$ and $m$.
\end{corollary}

For comparison, Browning, Le~Boudec and Sawin \cite[Lemma~3.19]{browning_le_boudec_sawin} showed that the number of primitive rank-$q$ lattices $\Lambda \subseteq \Z^{2m}$ with successive minima $\lambda_1 \le \cdots \le \lambda_q$ in dyadic ranges $\lambda_i \in [\xi_i,2\xi_i]$ is 
 \[
    \ll
    \prod_{i=1}^q
    \xi_i^{2m+1+q-2i}.
 \]
 Adapting the argument below to their setting would produce the upper bound
 \[
    \ll_\varepsilon
    \xi_q^\varepsilon
    \prod_{i=1}^q
    \xi_i^{2m+1+q-2i},
 \]
 which is weaker by a factor of $\xi_q^\varepsilon$.

Let us specialise the results introduced in Section~\ref{section:rationalpoints_fmt} to our setting. For $1 \le r \le m$, the rational points of the 
\emph{isotropic flag variety} $\mathcal{F}_r^\omega$ are given as 
\[
    \mathcal{F}_r^\omega(\mathbb{Q}) = \{ 
    (V_1,\ldots,V_r)\!:\,
    V_1 \subset \cdots \subset V_r \subseteq \mathbb{Q}^{2m},\, \dim(V_i)=i,~V_i \text{ isotropic, $i=1,\ldots,r$} 
    \}.
\]
The group $\Sp_m$ acts transitively on $\mathcal{F}_r^\omega$;
the stabiliser $P$ of the standard isotropic flag is a
$\Q$-rational parabolic subgroup of $\Sp_m$, and hence
\[
    \mathcal{F}_r^\omega \simeq \Sp_m/P.
\]
For $i=1,\ldots,r$, we denote by 
\[
    p_i: \mathcal{F}_r^\omega \to \mathbf{Gr}(i,2m), \qquad (V_1,\ldots,V_r) \mapsto V_i
\]
the projection onto the Grassmannian variety $\mathbf{Gr}(i,2m)$ and we set 
$    L_i = p_i^*\mathcal{O}_{\mathbf{Gr}(i,2m)}(1),
$
where $\mathcal{O}_{\mathbf{Gr}(i,2m)}(1)$ denotes the Pl\"ucker line bundle on the Grassmannian $\mathbf{Gr}(i,2m)$. Then the  
anticanonical divisor on $\mathcal{F}_r^\omega$ (see, for instance, \cite[Lemma~7.6]{feigin_finkelberg_littelmann}, with $d_i=i$) is
\[
    -K_{\mathcal{F}_r^\omega} = 2 L_1 + \cdots + 2 L_{r-1} + 2(m-r+1)L_r,
\]
and, choosing the heights $H_{L_i}$ associated with the Euclidean
norm on the Pl\"ucker coordinates, the associated
anticanonical height of a flag $(V_1,\ldots,V_r)$ is 
\[
    H_{-K_{\mathcal{F}_r^\omega}}(V_1,\ldots,V_r) = \prod_{j=1}^{r-1}\left( H_{L_j}(V_j)^2 \right) \, H_{L_r}(V_r)^{2(m-r+1)}.
\]

There is a bijection between flags of rational isotropic vector spaces $ V_1 \subset \cdots \subset V_r$ with $\dim(V_i)=i$ for $i=1,\ldots,r$ and flags of  primitive isotropic lattices $\Sigma_1 \subset \cdots \subset \Sigma_r$ with $\rank(\Sigma_i)=i$ for $i=1,\ldots,r$; indeed, this correspondence is given by 
\[
    (V_1,\ldots,V_r) \mapsto (V_1 \cap \Z^{2m}, \ldots, V_r\cap \Z^{2m})=: (\Sigma_1,\ldots,\Sigma_r).
\]
We have
$
    H_{L_i}(V_i) = \det(\Sigma_i);
$
see \cite[Chapter~I, Corollary~5I]{schmidt_diophantine_equations}.
Therefore the anticanonical height on $\mathcal{F}_{r}^\omega$ is 
\begin{equation}\label{eq:height_on_isoflagvar}
     H_{-K_{\mathcal{F}_r^\omega}}(V_1,\ldots,V_r) = \prod_{j=1}^{r-1}\left( \det(\Sigma_j)^2 \right) \, \det(\Sigma_r)^{2(m-r+1)}.
\end{equation}
Since $\mathcal{F}_r^\omega \simeq \Sp_m/P$ is a generalised flag
variety of the type considered in Section~\ref{section:rationalpoints_fmt},
we may apply \eqref{eq:rat_pts_upperbound_fmt}, i.e.~the bound
\begin{equation*}
    \#\{\x\in \mathcal{F}_r^\omega(\Q): H_{-K_{\mathcal{F}_r^\omega}}(\x)\le B\}
    \ll_\varepsilon B^{1+\varepsilon}.
\end{equation*}
We are now in a position to prove the following.

\begin{proposition}\label{prop:countingflags}
    Let $0\le q\le m$ and $0\le s\le m-q$ be integers and let $1 \le \theta_1 \le \cdots \le \theta_q \le T$ and 
    $0 < \rho_1 \le \cdots \le \rho_s \le T$ be positive real numbers.
    Then the number of pairs of lattices $(\Sigma, \Xi)$ such that 
    \begin{enumerate}
        \item $\Sigma \subseteq \Z^{2m}$ is a primitive isotropic lattice of rank $q$ with successive minima $\sigma_1 \le \cdots \le \sigma_q $ such that $ \theta_i \le \sigma_i \le 2 \theta_i$ for $1\le i \le q$, and 
        \item $\Xi \subseteq \Gamma_\Sigma / \Sigma$, with $\Gamma_\Sigma$ defined in \eqref{eq:def_gamma_sigma}, is a primitive isotropic (with respect to $\overline{\omega}$) lattice of rank $s$ with successive minima $\xi_1 \le \cdots \le \xi_s $ such that $ \rho_i \le \xi_i \le 2 \rho_i$ for $1 \le i \le s$,
    \end{enumerate}
    can be bounded by 
    \[
    \# 
    \left\{
    (\Sigma, \Xi) \text{ with \emph{(1)} and \emph{(2)}} 
    \right\} \ll_\varepsilon T^\varepsilon \prod_{i=1}^{q} \theta_i^{2(m-i+1)} 
    \prod_{i=1}^{s} \rho_i^{2(m-q-i+1)}.
    \]
    The implicit constant only depends on $\varepsilon$ and $m$.
\end{proposition}

\begin{proof}
The case $q = s = 0$ is trivial, so let us assume $q + s \ge 1$ from now on. 
    We begin by associating a flag of lattices to each pair $(\Sigma, \Xi)$ with (1) and (2) as follows.     
    By Lemma~\ref{lem:basis_successive_min_size}, we may choose once and for all a pair of bases $((\x_i)_{i=1}^q,(\y_j)_{j=1}^s)$ such that 
    \[
        |\x_i| \ll \theta_i \text{ for $i =1,\ldots,q$ }\quad 
        \text{ and } \quad
        |\y_j| \ll \rho_j \text{ for $j=1,\ldots,s$.}
    \]
    Setting
    \[
    \Sigma_i=\mathrm{span}_{\Z} (\x_1,\ldots,\x_i) \quad \text{ and } \quad 
    \Xi_j = \mathrm{span}_{\Z} (\y_1,\ldots,\y_j)
    \]
     yields flags of primitive isotropic lattices
    \[
        \Sigma_1 \subset \cdots \subset \Sigma_q = \Sigma, \qquad \Xi_1 \subset \cdots \subset \Xi_s = \Xi,
    \] 
    with the obvious omissions when $q=0$ or $s=0$.
    Indeed, any lattice spanned by a subset of a basis of a primitive isotropic lattice is primitive and isotropic. 
    Using the lifting construction from Lemma~\ref{lem:lifting}, we obtain a flag of primitive isotropic lattices
\[
\Sigma_1\subset\cdots\subset\Sigma_q
\subset
\Xi_1^\uparrow\subset\cdots\subset\Xi_s^\uparrow
\subset\Z^{2m}, \quad \rank (\Sigma_i)=i, \quad \rank(\Xi_j^\uparrow) = q+j
\]
for $i =1,\ldots, q$ and $j=1,\ldots, s$.
The map $(\Sigma,\Xi) \mapsto \Sigma_1\subset\cdots\subset\Sigma_q
\subset
\Xi_1^\uparrow\subset\cdots\subset\Xi_s^\uparrow$ is injective, since $\Sigma = \Sigma_q$ and $\Xi = \Xi^\uparrow_s/\Sigma_q$, so it suffices to count the number of possible flags.

By Minkowski's second theorem \eqref{eq:minkowski} and the identity $\det(\Xi_j^\uparrow)
=
\det(\Sigma)\det(\Xi_j)$, we have 
\[
\det(\Sigma_i) \ll \prod_{k=1}^i \theta_k  \text{ for $i=1,\ldots,q$},
\quad
\det(\Xi_j^\uparrow)
\ll
\bigg(\prod_{i=1}^q\theta_i\bigg)
\bigg(\prod_{k=1}^j\rho_k\bigg) \text{ for $j=1,\ldots,s$}.
\]
The flag $\Sigma_1\subset\cdots\subset\Sigma_q
\subset
\Xi_1^\uparrow\subset\cdots\subset\Xi_s^\uparrow$ is a rational point of  $\mathcal{F}_{q+s}^\omega$. Hence, by
\eqref{eq:height_on_isoflagvar}, its anticanonical height is at most
\[
\ll
\prod_{i=1}^q \theta_i^{2(m-i+1)}
\prod_{j=1}^s \rho_j^{2(m-q-j+1)}.
\]
Indeed, $\theta_i$ occurs in the bound for the determinants of
$\Sigma_i,\ldots,\Sigma_{q}, \Xi_1^\uparrow, \ldots, \Xi_s^\uparrow$ and therefore has total exponent
\[
2(q+s-i)+2(m-q-s+1)=2(m-i+1),
\]
while $\rho_j$ occurs in the bound for the determinants of
$ \Xi_j^\uparrow, \ldots, \Xi_s^\uparrow$ and therefore has total exponent
\[
2(s-j)+2(m-q-s+1)=2(m-q-j+1).
\]
    
    An application of \eqref{eq:rat_pts_upperbound_fmt} now yields that the number of such flags, and thus the number of tuples $(\Sigma, \Xi)$, 
    is  
    \[
    \ll_\varepsilon T^\varepsilon \prod_{i=1}^q \theta_i^{2(m-i+1)}
\prod_{j=1}^s \rho_j^{2(m-q-j+1)}.
    \]
    Here we used the fact that $\theta_i\ll T$ for $i =1,\ldots, q$ and $\rho_j \ll T$ for $ j = 1,\ldots, s$.
\end{proof}

\subsection{Counting isotropic tuples}\label{section:tuples}
We now establish the uniform upper bound for the number of isotropic tuples of integer
vectors lying in prescribed sets. 

\begin{proposition}\label{prop:translated_isotropic_frames}
    Let $1 \le r \le m$ be a positive integer and let $
    D_1 \ge \cdots \ge D_r \ge 1
    $. For $ i =1,\ldots,r$, let $\mathcal{D}_i \subseteq \R^{2m}$ be a set of diameter $\le D_i$, and set 
    \[
    \mathcal{C}(\mathcal{D}_1,\ldots,\mathcal{D}_r) =
    \bigg\{\!
    (\x_1, \ldots, \x_r)\! \in\! 
    \prod_{i=1}^r 
    \left(
    \mathcal{D}_i \cap \Z^{2m}
    \right): \omega(\x_i, \x_j)=0 \text{ for all } 1 \le i,j \le r
    \bigg\},
    \]
    with $\omega$ defined in \eqref{def:omega}. 
     Then 
    \[
    \# 
    \mathcal{C}(\mathcal{D}_1,\ldots,\mathcal{D}_r)
    \! \ll_\varepsilon \! D_1^\varepsilon \prod_{i=1}^r D_i^{2(m-i+1)},
    \]
    where the implicit constant only depends on $\varepsilon$ and $m$. 
\end{proposition}

\begin{proof}
    To start with, we inductively define an injection 
    \[
    \mathcal{C}(\mathcal{D}_1,\ldots,\mathcal{D}_r) \to \mathcal{C}([-D_1,D_1]^{2m},\ldots,[-D_r,D_r]^{2m}).
    \]
    For any tuple $(\x_2,\ldots,\x_r)$
    such that there exists $\x_1$ with 
    $(\x_1,\ldots,\x_r) \in \mathcal{C}(\mathcal{D}_1, \ldots, \mathcal{D}_r)$, we choose, once and for all, $\cc = \cc(\x_2,\ldots,\x_r)$ with $(\cc,\x_2,\ldots,\x_r) \in \mathcal{C}(\mathcal{D}_1, \ldots, \mathcal{D}_r)$. 
    Then the map 
    \begin{align*}
        i_1:  \mathcal{C}(\mathcal{D}_1,\ldots,\mathcal{D}_r) \to \mathcal{C}([-D_1,D_1]^{2m},\mathcal{D}_2,\ldots,\mathcal{D}_r)\\
        (\x_1,\x_2,\ldots,\x_r) \mapsto
        (\x_1-\cc (\x_2,\ldots,\x_r),\x_2,\ldots,\x_r) 
    \end{align*}
 is clearly injective.
 Moreover, it is well-defined. Indeed, for $i=2,\ldots,r$, by bilinearity of $\omega$ we have 
 \[
    \omega(\x_1 - \cc(\x_2,\ldots,\x_r), \x_i) = \omega(\x_1,\x_i) - \omega(\cc(\x_2,\ldots,\x_r), \x_i) = 0,
 \]
while we have $\omega(\y,\y)=0$ for any $\y \in \R^{2m}$, and thus in particular 
\[
   \omega(\x_1 - \cc(\x_2,\ldots,\x_r) , \x_1 - \cc(\x_2,\ldots,\x_r) ) =0.
\]
 
 For $2 \le k \le r$, we may similarly define an injective map $i_k$ between 
 $$ \mathcal{C}([-D_1,D_1]^{2m},\ldots,
 [-D_{k-1},D_{k-1}]^{2m},
 \mathcal{D}_k, \ldots,
 \mathcal{D}_r)$$ and  $$\mathcal{C}([-D_1,D_1]^{2m},\ldots,
 [-D_{k},D_{k}]^{2m},
 \mathcal{D}_{k+1}, \ldots,
 \mathcal{D}_r).
$$
Hence $i_r \circ \cdots \circ i_1$ gives the desired injective map, and so it now suffices to bound 
\[
  \#  \mathcal{C}([-D_1,D_1]^{2m},\ldots,[-D_r,D_r]^{2m}).
\]

    For any   $(\y_1,\ldots,\y_r) \in  \mathcal{C}([-D_1,D_1]^{2m},\ldots,[-D_r,D_r]^{2m})$, consider the primitive lattice 
    \[
    \mathrm{span}_{\R}(\y_1,\ldots,\y_r) \cap \Z^{2m}.
    \]
    Since the $\y_i$ are pairwise isotropic, this is an isotropic lattice of rank $s$ for some $s\le r$;  denote its successive minima by $\sigma_1 \le \cdots \le \sigma_s$. Since $|\y_i|\le D_i$ for all $i \le r$, and since the $D_i$ are decreasing, we observe that 
    \[
    \sqrt{2m}D_1 \ge \sigma_s, \sqrt{2m} D_2 \ge \sigma_{s-1}, \ldots, \sqrt{2m} D_{s} \ge \sigma_1.
    \]

    We now decompose 
    the set  \[
    \left\{ 
    (\mu_1,\ldots,\mu_s)\in
    \left[1,\sqrt{2m}D_{s} \right]\times \cdots \times \left[1,\sqrt{2m} D_1\right]: \mu_1 \le \cdots \le \mu_s 
    \right\}
    \]
    into dyadic ranges; fix one such range, say $\left[d_s,2d_s\right]\times \cdots\times\left[d_1,2d_1\right]$ for some $d_s \le \cdots \le d_1$. By Corollary~\ref{cor:succminima}, the number of isotropic lattices with successive minima $d_{s-i+1} \le \sigma_i \le 2d_{s-i+1}$ for $i=1,\ldots,s$ is bounded by 
    \[
    \ll_\varepsilon D_1^{\varepsilon/2} \prod_{i=1}^s d_{s-i+1}^{2(m-i+1)} = D_1^{\varepsilon/2} \prod_{k=1}^s d_{k}^{2(m-s+k)}.
    \]
    Fixing one such lattice $\Sigma$, the number of $(\x_1,\ldots,\x_r) \in \Sigma^r$ with $|\x_i| \le D_i$ for $i =1,\ldots,r$ can be bounded by Lemma~\ref{lem:latticepointboundbox} as 
    \[
    \ll
    \prod_{i=1}^r
    \prod_{j=1}^s 
        \left( 
        1 + \frac{D_i}{d_j}
        \right).
    \]
    Therefore the total contribution to $\# \mathcal{C}(\mathcal{D}_1,\ldots,\mathcal{D}_r)$ from this dyadic range is 
    \begin{equation}\label{eq:dydic_contr_to_bound}
            \ll_\varepsilon 
    D_1^{\varepsilon/2} \prod_{k=1}^s d_{k}^{2(m-s+k)}
    \prod_{i=1}^r
    \prod_{j=1}^s 
        \left( 
        1 + \frac{D_i}{d_j}
        \right).
    \end{equation}
It remains to argue that this contribution is 
\begin{equation}\label{eq:remainstoshow}
\ll_\varepsilon  D_1^{\varepsilon/2} \prod_{i=1}^s D_i^{2(m-i+1)}.
\end{equation}
Then the claimed bound will follow by summing over the $\ll_\varepsilon D_1^{\varepsilon/2}$ dyadic intervals. In the case $s=0$, the products are empty and so the bound holds. Thus we may assume $s\ge 1$ from now on.
    
   We first deal with the factors coming from $i \le s$. Since the $d_i$ form a descending sequence and $d_i \ll D_i$, we get
   \begin{align*}
          \prod_{i=1}^s
    \prod_{j=1}^s 
        \left( 
        1 + \frac{D_i}{d_j}
        \right) &= 
        \prod_{i=1}^s
    \prod_{j=1}^{i} 
        \left( 
        1 + \frac{D_i}{d_j}
        \right)
         \prod_{j=i+1}^{s} 
        \left( 
        1 + \frac{D_i}{d_j}
        \right) \\
        &\ll
        \prod_{i=1}^s
    \prod_{j=1}^{i} 
        \left( 
        \frac{D_i}{d_i}
        \right)
         \prod_{j=i+1}^{s} 
        \left( 
        \frac{D_i}{d_j}
        \right) 
        = \prod_{i=1}^s \bigg(\frac{D_i^s}{d_i^i}
        \prod_{j=i+1}^s 
        \frac{1}{d_j}
        \bigg) = \prod_{i=1}^s \bigg(\frac{D_i^s}{d_i^{2i-1}}\bigg).
   \end{align*}
   For the products with $i>s$, we observe that $D_i \le D_j$ and so 
   \[
      \prod_{i=s+1}^r
    \prod_{j=1}^s 
        \left( 
        1 + \frac{D_i}{d_j}
        \right) \ll 
        \prod_{i=s+1}^r
    \prod_{j=1}^s 
        \left( 
        \frac{D_j}{d_j}
        \right) \ll
         \prod_{j=1}^s 
        \frac{D_j^{r-s}}{d_j^{r-s}}
   \]
   Combining these bounds gives by \eqref{eq:dydic_contr_to_bound} 
   \begin{align*}
   \ll_\varepsilon D_1^{\varepsilon/2 } \prod_{k=1}^s d_{k}^{2(m-s+k)} \prod_{k=1}^s \frac{D_k^{r}}{d_{k}^{2k-1+r-s}} 
   &= 
   D_1^{\varepsilon/2 }
   \prod_{k=1}^s 
   \left( D_k^r
   d_k^{2m-s-r+1}
   \right)
   \\
   &\ll_\varepsilon
    D_1^{\varepsilon/2 }
   \prod_{k=1}^s 
   \left( D_k^r
   D_k^{2m-s-r+1}
   \right)
   \\
   &= 
   D_1^{\varepsilon/2 }
   \prod_{k=1}^s 
   \left(
   D_k^{2m-s+1}
   \right) .
   \end{align*}

   Since the sequences $\log(D_i)$ and $2(m-i+1)$ are both decreasing, the Chebyshev sum inequality yields  
   \[
   \sum_{i=1}^s 2(m-i+1)\log(D_i) \ge \frac{1}{s} \sum_{i=1}^s 2(m-i+1) 
   \sum_{j=1}^s \log(D_j) = (2m-s+1) \sum_{j=1}^s \log(D_j). 
   \]
   Thus we finally get 
   \[
   \prod_{k=1}^s 
   D_k^{2m-s+1}
    \le \prod_{i=1}^s 
   D_i^{2(m-i+1)},
   \]
   which completes the proof of \eqref{eq:remainstoshow}.
\end{proof}

\section{The uniform upper bound}

In this section, we prove Theorem~\ref{thm:upperbound}. 
More concretely, in Section~\ref{section:associ}, we begin by associating to each matrix $U\in\mathfrak{B}_{\mathbf K}(T)\cap\Sp_m(\Z)$ a primitive isotropic lattice $\Sigma_U$ encoding the short directions of the lattice $\Lambda_U$ generated by the first $m$ columns of $U$. We show that the factor lattice $\Lambda_U/\Sigma_U$ has no short nonzero vectors.
 
In Section~\ref{section:countinginsydaicrange}, we fix the rank of $\Sigma_U$ and dyadic ranges for its successive minima, and prove a uniform bound for the number of matrices in each such range. We first count the possible quotient data $(\Sigma_U,(\mathbf w_i)_{i=1}^m)$ and then bound the number of matrices giving rise to a fixed tuple. Finally, in Section~\ref{section:completingproof}, we sum over the dyadic ranges and complete the proof of Theorem~\ref{thm:upperbound}.

\subsection{Associating lattices to symplectic matrices}\label{section:associ}

Let $ U \in \mathcal{M}_{2m}(\Z)$ be a symplectic matrix. We write 
\[
    U = \left(
    \x_1 ~ \cdots ~ \x_m ~ \y_1 ~ \cdots ~ \y_m
    \right),
\]
with $\x_1, \ldots, \x_m, \y_1, \ldots, \y_m$ denoting the columns of $U$; since $U$ is invertible, these vectors are linearly independent.
In this section, we study the properties of the lattice 
\begin{equation*}
        \Lambda_U := \mathrm{span}_\Z(
    \x_1, \ldots, \x_m
    ).
\end{equation*}
We denote its successive minima by $\lambda_1 \le \cdots \le \lambda_m$.

Furthermore, let 
 \begin{equation}\label{eq:choice_of_A}
     A > 8 m \left( 1+ C_{2m} \right),
 \end{equation}
  where $C_{2m}$ is the constant in Lemma~\ref{lem:close_to_vector}, be a fixed parameter and set 
\begin{equation}
\label{eq:choiceofq}
        q = 
     \max \{ i:~1 \le i \le m,~ \lambda_i \le   A^i T
        \},
\end{equation}
with the convention that $q= 0$ if the set $\{ i:~1 \le i \le m,~ \lambda_i \le  A^i T\}$ is empty. We define 
\begin{equation}\label{eq:def_SigmaU}
       \Sigma_U = 
   \Lambda_U \cap \mathrm{span}_{\R}\left( \{ 
    \x \in \Lambda_U:~ |\x| \le  A^q T
    \}
    \right).
\end{equation}

\begin{lemma}
    The lattice $\Lambda_U$ is primitive and Lagrangian.
\end{lemma}
\begin{proof}
By \eqref{eq:def_syplectic_matrix}, we have 
\[
    U^t \Omega \, U = \Omega,
\]
and therefore 
\[
    \omega(\x_i, \x_j) = 0 \text{ for all $1 \le i,j \le m$}.
\]
This shows that $\Lambda_U$ is isotropic. Since the vectors are linearly independent, it has rank $m$, and so it is
Lagrangian.

It remains to show that $\Lambda_U$ is primitive. Indeed, since $U^t \Omega \, U = \Omega$, we get 
\[
    \det(U)^2 \det(\Omega) = \det(U^t \Omega U) = \det(\Omega),
\]
so $\det(U) \in \{ \pm 1 \}$. This implies that $\Lambda_U = \mathrm{span}_{\Z}(\x_1,\ldots,\x_m)$ is also primitive.
\end{proof}

We now turn to studying $\Sigma_U$.

\begin{lemma}\label{lem:nonzerolowerbound}
    The lattice $\Sigma_U$ has rank $q$, and it is primitive and isotropic. Its successive minima, denoted by $\sigma_1 \le \cdots \le \sigma_q$, satisfy the inequality
    \[
    \sigma_i \le  A^q T.
    \]
    Moreover, all nonzero $\x \in (\Lambda_U / \Sigma_U)\setminus \{ \mathbf{0} \}$ satisfy 
    $
        |\x| > 2\sqrt{2m} T.
    $
\end{lemma}
\begin{proof}
    We first prove that $\rank (\Sigma_U) = q$. Indeed, $\lambda_1 \le \cdots \le \lambda_q \le A^q T$, and therefore the set 
\[
      \{ 
    \x \in \Lambda_U:~ |\x| \le  A^q T
    \}
\]
contains $q$ linearly independent vectors.
If $m=q$, these $q$ linearly independent vectors already show that $\rank(\Sigma_U) = q$. If $m>q$, we further observe that 
\[
    \{ 
    \x \in \Lambda_U:~ |\x| \le  A^q T
    \} \subseteq B(\mathbf{0}, \sqrt{2m} A^q T),
\]
but the $(q+1)$-th successive minimum of $\Lambda_U$
satisfies 
\[
    \lambda_{q+1} > A^{q+1} T > \sqrt{2m} A^q T.
\]
Thus $  \{ 
    \x \in \Lambda_U:~ |\x| \le  A^q T
    \}$ is contained in a space of dimension $q$, and hence 
    \[
    \rank (\Sigma_U) =q.
    \]

The lattice $\Sigma_U$ is isotropic since it is contained in the isotropic lattice $\Lambda_U$.
By construction,
the lattice 
\[
\Sigma_U = 
   \Lambda_U \cap \mathrm{span}_{\R}\left( \{ 
    \x \in \Lambda_U:~ |\x| \le  A^q T
    \}
    \right)
\]
 is primitive in the primitive lattice $\Lambda_U$, and Lemma~\ref{lem:primitivityfactsetc} implies that it is primitive in $\Z^{2m}$. 

For $i \le q$, we have $\lambda_i \le \lambda_q \le A^q T$, 
so intersecting a ball of radius $\lambda_i$, centred at the origin, with 
\[
    \{ 
    \x \in \Lambda_U:~ |\x| \le  A^q T
    \} 
\]
yields at least $i$ linearly independent vectors.
Hence $\sigma_i \le \lambda_i \le A^q T$.

It remains to prove that $|\x| >2 \sqrt{2m} T$ for all $\mathbf{0} \neq \x \in \Lambda_U / \Sigma_U$. If $q=m$, then $\Lambda_U = \Sigma_U$ and so $ \Lambda_U / \Sigma_U = \{ \mathbf{0}\}$, so there is nothing to prove. If $q=0$, then $\lambda_1>AT$ and $\Sigma_U=\{\mathbf{0}\}$. Since \(A > \sqrt{2m}\) is chosen sufficiently large, the lemma statement follows.

Therefore we may assume $1\le q<m$. Suppose that there is $\mathbf{0} \neq \x \in \Lambda_U / \Sigma_U$ with $|\x| \le 2 \sqrt{2m} T$. Write $\x = \pi_{\Sigma_U}(\w)$ for some $\w \in \Lambda_U$. Then 
\[
    \w - \pi_{\Sigma_U}( \w )  \in \mathrm{span}_{\R}(\Sigma_U).
\]
Since all successive minima of $\Sigma_U$ are of size at most $A^q T$,
 an application of Lemma~\ref{lem:close_to_vector} yields that there is some  $\w_0 \in \Sigma_U$ with 
\[
    |\w - \pi_{\Sigma_U}(\w) - \w_0| \le C_{2m} A^q T. 
\]
By the triangle inequality and our choice of $A$ in \eqref{eq:choice_of_A}, we get 
\[
    |\w - \w_0| \le |\pi_{\Sigma_U}(\w)| + | \w -  \pi_{\Sigma_U}(\w) - \w_0| \le 2 \sqrt{2m}T + C_{2m} A^q T \le (2 \sqrt{2m}+C_{2m})A^q T 
\]
and thus 
\[
     \|\w - \w_0\|_2 \le \sqrt{2m} (2 \sqrt{2m}+C_{2m})A^q T  < A^{q+1}T.
\]
Moreover, by $\pi_{\Sigma_U}(\w - \w_0) = \pi_{\Sigma_U}(\w) \neq \mathbf{0}$, we know that $\w - \w_0\not \in \mathrm{span}_{\R}(\Sigma_U)$. Since $\sigma_q \le A^qT$, it follows that $\Sigma_U \cup \{\w-\w_0 \} \subseteq \Lambda_U$ contains $q+1$ linearly independent vectors of size at most $A^{q+1}T$.
This contradicts our choice of $q$ in \eqref{eq:choiceofq}.
\end{proof}

\subsection{Counting in a fixed dyadic range}\label{section:countinginsydaicrange}
Let us fix a positive integer $q \in \{1,\ldots, m\}$ 
as well as real numbers 
\[
    1 \le \theta_1 \le \cdots \le \theta_q \le A^m T,
\]
where we set $\boldsymbol{\theta} := (\theta_1,\ldots, \theta_q)$.
In this section we bound the number of symplectic integer matrices $U = \left(
    \x_1 ~ \cdots ~ \x_m ~ \y_1 ~ \cdots ~ \y_m
    \right) \in \mathfrak{B}_{\mathbf K}(T)\cap \Sp_{m}(\Z)$ such that the lattice $\Sigma_U$, as defined in \eqref{eq:def_SigmaU}, has rank $q$ and successive minima $\sigma_1 \le \cdots \le \sigma_q$ with 
\[
    \theta_i \le \sigma_i \le 2 \theta_i \quad \text{ for }i=1,\ldots,q;
\]
we denote the set of all such matrices by $\mathcal{A}_T(\boldsymbol{\theta})$. The main goal of this section is to prove the following result.

\begin{proposition}\label{prop:ATbound}
    For every $q\in\{1,\ldots,m\}$ and every admissible $\boldsymbol{\theta}$, we have
\[ 
\#\mathcal{A}_T(\boldsymbol\theta) \ll_\varepsilon T^{m(m+1)+\varepsilon}.
\]
\end{proposition}

We decompose our proof of Proposition~\ref{prop:ATbound} into two steps. First, we give an upper bound for the number of pairs $(\Sigma, (\w_i)_{i=1}^m)$ such that there exists some symplectic integer matrix $U =  \left(
    \x_1 ~ \cdots ~ \x_m ~ \y_1 ~ \cdots ~ \y_m
    \right) \in \mathcal{A}_T(\boldsymbol{\theta})$  with \begin{equation*}
    \Sigma = \Sigma_U \quad  \text{and} \quad \w_i = \pi_{\Sigma_U}(\x_i)
    \text{ for } i=1,\ldots,m.
    \end{equation*}
    Then we count the number of admissible $U$ corresponding to one tuple $(\Sigma,(\w_i)_{i=1}^m)$.
    We proceed with proving the following auxiliary result.

\begin{lemma}\label{lem:determininglambdamodsigma}
    Assume that $U \in \mathcal{A}_T(\boldsymbol{\theta})$.
    Then the lattice $\Lambda_U / \Sigma_U$ is a primitive lattice of rank $m-q$ in $\pi_{\Sigma_U}(\Z^{2m})$,
    generated by the projected columns
    \[
        \pi_{\Sigma_U}(\x_1),\ldots,\pi_{\Sigma_U}(\x_m).
    \]
    Moreover, once $\Sigma_U$ and $m-q$ linearly independent vectors among
    $\pi_{\Sigma_U}(\x_1),\ldots,\pi_{\Sigma_U}(\x_m)$ are fixed, all the remaining projected
    columns $\pi_{\Sigma_U}(\x_j)$ are uniquely determined.
\end{lemma}
\begin{proof}
Since $\Sigma_U$ is primitive in $\Lambda_U$ and $\Lambda_U$ is primitive in
$\Z^{2m}$, Lemma~\ref{lem:primitivityfactsetc} implies that
$\Lambda_U/\Sigma_U$
is primitive in
$\pi_{\Sigma_U}(\Z^{2m})$.
Moreover, we have
\[
\rank(\Lambda_U/\Sigma_U)
=
\rank(\Lambda_U)-\rank(\Sigma_U)
=
m-q.
\]
Since
\[
\Lambda_U=\mathrm{span}_{\Z}(\x_1,\ldots,\x_m),
\]
we also have
\[
\Lambda_U/\Sigma_U
=
\mathrm{span}_{\Z}\big(
\pi_{\Sigma_U}(\x_1),\ldots,\pi_{\Sigma_U}(\x_m)
\big).
\]

Now fix $m-q$ linearly independent projected columns, say
$\w_1,\ldots,\w_{m-q}$.
Then 
\[
    \mathrm{span}_{\R}(\w_1,\ldots,\w_{m-q}) \cap \pi_{\Sigma_U}(\Z^{2m}) = \Lambda_U/\Sigma_U,
\]
since $\Lambda_U/\Sigma_U$ is primitive of rank $m-q$ and contains
the $\w_i$. This proves that these vectors determine the factor lattice
$\Lambda_U/\Sigma_U$.

Let $\x_j$ be any remaining column. Its projection
$\pi_{\Sigma_U}(\x_j)$ lies in the fixed lattice $\Lambda_U/\Sigma_U$ and also
in the projection of the corresponding column box; since the box has diameter at most
$
2\sqrt{2m}T$, so does its projection.
On the other hand, by Lemma~\ref{lem:nonzerolowerbound}, every nonzero vector of
$\Lambda_U/\Sigma_U$ has norm strictly larger than $2\sqrt{2m}T$.
This means that the projected column box contains at most one point of
$\Lambda_U/\Sigma_U$, and thus $\pi_{\Sigma_U}(\x_j)$ is uniquely determined.
\end{proof}

\begin{lemma}\label{lem:bounding_sigmaw}
    The number of pairs $(\Sigma, (\w_i)_{i=1}^m)$ such that there exists some symplectic integer matrix $U =  \left(
    \x_1 ~ \cdots ~ \x_m ~ \y_1 ~ \cdots ~ \y_m
    \right) \in \mathcal{A}_T(\boldsymbol{\theta})$  with \begin{equation}\label{eq:cond_on_tupleSigmaw}
    \Sigma = \Sigma_U \quad  \text{and} \quad \w_i = \pi_{\Sigma_U}(\x_i)
    \text{ for } i=1,\ldots,m
    \end{equation} can be bounded from above by
    \[
        \# \{ 
        (\Sigma, (\w_i)_{i=1}^m) \text{ satisfying \eqref{eq:cond_on_tupleSigmaw}}
        \}
        \ll_\varepsilon T^{ (m-q)(m-q+1) +\varepsilon}  \prod_{i=1}^q \theta_i^{2(m-i+1)}.
    \]
\end{lemma}
\begin{proof}
By Lemma~\ref{lem:determininglambdamodsigma}, each admissible tuple $(\Sigma, (\w_i)_{i=1}^m)$ is determined by $\Sigma$ and any $m-q$ linearly independent vectors among the $ \w_i$. Since there are $\ll 1$ possible choices of indices, it therefore actually suffices to give an upper bound on the number of tuples 
\[
    (\Sigma, (\w_{j_i})_{i=1}^{m-q}), \quad 1\le j_1 < \ldots < j_{m-q} \le m.
\]
For notational convenience, we relabel $j_i=i$ for $1 \le i \le m-q$.

For fixed $\Sigma$, every $\w_i$ is of the form $\w_i = \pi_\Sigma(\x_i)$ for some integer vector $\x_i$ that belongs to the set of possible $i$-th columns of matrices in $\mathfrak{B}_{\mathbf{K}}(T)$. This condition constrains the $\x_i$ to a set of diameter $\le 2\sqrt{2m}T$; thus $\pi_\Sigma (\x_i) = \w_i$ also lies in a set of diameter at most $2 \sqrt{2m} T$. 

Moreover, the vectors $\w_1,\ldots,\w_{m-q}$ are pairwise isotropic with
respect to $\overline{\omega}$. Applying the same translation
argument as at the beginning of the proof of Proposition~\ref{prop:translated_isotropic_frames},
now inside $\Gamma_\Sigma/\Sigma$, gives an injection into isotropic tuples
\[
(\mathbf{v}_1,\ldots,\mathbf{v}_{m-q})\in(\Gamma_\Sigma/\Sigma)^{m-q}
\]
satisfying
\[
|\mathbf{v}_i|\ll T\qquad \text{for }i=1,\ldots,m-q.
\]
So, for fixed $\Sigma$, it suffices to bound the number of such tuples $(\mathbf{v}_1,\ldots,\mathbf{v}_{m-q})$.

Let us define 
\[
    \Xi = \mathrm{span}_{\R}(\mathbf{v}_1,\ldots,\mathbf{v}_{m-q})
    \cap \Gamma_\Sigma/\Sigma,
\]
a primitive isotropic lattice of rank $s \le m-q$, and denote its successive minima by $\xi_1 \le \cdots \le \xi_s$. 
The case $s=0$ now follows by Corollary~\ref{cor:succminima}, so we may assume $s\ge1$.
Since 
$\mathbf{v}_1,\ldots,\mathbf{v}_{m-q} \in \Xi$ and
$|\mathbf{v}_i| \ll T$, it follows that $\xi_1 \le \cdots \le \xi_s \ll T$. By Minkowski's second theorem  \eqref{eq:minkowski} and $\sigma_i \ll T$, we have $T^{-q} \ll (\det(\Sigma))^{-1} \ll \xi_1$.

We now restrict to lattices $\Xi \subseteq \Gamma_\Sigma/\Sigma$ with successive minima inside dyadic intervals $T^{-q} \ll \rho_i \le \xi_i \le 2 \rho_i$ with $\rho_1 \le \cdots \le \rho_{s} \ll T$. 
By Proposition~\ref{prop:countingflags}, the number of possible tuples $(\Sigma,\Xi)$
is 
\[
    \ll_\varepsilon
    T^{\varepsilon/2}
\prod_{i=1}^q\theta_i^{2(m-i+1)}
\prod_{k=1}^s\rho_k^{2(m-q-k+1)}.
\]
The first product accounts for the possible lattices $\Sigma$, while the second accounts for the possible lattices $\Xi$ inside the factor lattice $\Gamma_\Sigma/\Sigma$. For each tuple $(\Sigma, \Xi)$, by Lemma~\ref{lem:latticepointboundbox}, the number of possible $\mathbf{v}_1,\ldots,\mathbf{v}_{m-q}$ is 
\[
    \ll \prod_{j=1}^s \left( 1 
    + \frac{T}{\rho_j}
    \right)^{m-q} \ll
    \prod_{j=1}^s \frac{T^{m-q}}{\rho_j^{m-q}} = T^{s(m-q)}  \prod_{j=1}^s {\rho_j^{-(m-q)}}.
\]
Therefore, the contribution to the final count from this dyadic interval is at most
\[
    \ll_\varepsilon
     T^{s(m-q)+\varepsilon/2}
\prod_{i=1}^q\theta_i^{2(m-i+1)}
\prod_{k=1}^s\rho_k^{m-q-2(k-1)}.
\]

It now suffices to show that
\[
\prod_{k=1}^s\rho_k^{m-q-2(k-1)}
    \ll T^{s(m-q-s+1)}.
\]
Then, since $s \le m-q$ and $T^{-q} \ll \rho_i \ll T$, the lemma statement follows by summing over the $\ll_\varepsilon T^{\varepsilon/2}$ dyadic ranges.

The sequence $\log(\rho_i)$ is increasing, whereas the sequence $m-q-2i+2$ is decreasing. Thus the reverse Chebyshev sum inequality yields 
\begin{align*}
    \sum_{i=1}^s (m-q-2i+2) \log(\rho_i) &\le \frac{1}{s} \sum_{i=1}^s \log(\rho_i) 
    \sum_{j=1}^s (m-q-2j+2) \\
    &= (m-q-s+1) \sum_{i=1}^s \log(\rho_i).
\end{align*}
Since $\rho_i \ll T$, this implies 
\[
\prod_{k=1}^s\rho_k^{m-q-2(k-1)} \ll \prod_{k=1}^s\rho_k^{m-q-s+1} \ll T^{s(m-q-s+1)}. 
\]
Summing over the $\ll1$ possible choices of $s$ and over the $\ll_\varepsilon T^{\varepsilon/2}$ dyadic ranges proves the lemma.
\end{proof}

\begin{lemma}\label{lem:boundingpersigmaw}
    Fix a primitive isotropic lattice $\Sigma \subseteq \Z^{2m}$ of rank $q$ and fix $(\w_i)_{i=1}^m \in (\Gamma_\Sigma / \Sigma)^m$. Then the number of $U =  \left(
        \x_1~\cdots~\x_m~
        \y_1~\cdots~\y_m
        \right) \in \mathcal{A}_T(\boldsymbol{\theta})$ such that 
        \begin{enumerate}
            \item $\Sigma_U = \Sigma$ and
            \item $\pi_{\Sigma}(\x_i)=\w_i
        $ for $i=1,\ldots,m$ 
        \end{enumerate}
        is  
        \[
            \# \{ U \text{ with \emph{(1)} and \emph{(2)}} \}  \ll_\varepsilon
    T^{q(2m-q+1)+\varepsilon}
    \prod_{i=1}^q
    \theta_i^{-2(m-i+1)}.
        \]
\end{lemma}

\begin{proof}
     Let $\Sigma \subseteq \Z^{2m}$ be a primitive isotropic lattice of rank $q$ and let $(\w_i)_{i=1}^m \in (\Gamma_\Sigma / \Sigma)^m$. 
     We begin by showing that the tuple $(\Sigma, (\w_i)_{i=1}^m)$ uniquely determines the lattice $\Lambda_U$.
     
     Consider the lattice 
     \[
        \Xi = \mathrm{span}_\Z(\w_1,\ldots,\w_m) \subseteq \Gamma_\Sigma / \Sigma.
     \]
     Then for every integer symplectic matrix $U$ satisfying (1) and (2), we have 
     \[
     \Lambda_U / \Sigma = \Xi.
     \]
     Let us furthermore define 
     \[
     \Xi^\uparrow = \pi_\Sigma^{-1}( \Xi)\cap \Gamma_\Sigma.
     \]
     We claim that $\Lambda_U = \Xi^\uparrow$. 
     Since $\Sigma \subseteq \Lambda_U$ and $\Lambda_U$ is isotropic, we have $\Lambda_U \subseteq \Gamma_\Sigma$;
     since $\Lambda_U = \mathrm{span}_\Z (\x_1, \ldots, \x_m)$ with $\w_i =\pi_\Sigma(\x_i)$, it is clear that $\x_1, \ldots, \x_m \in   \pi_\Sigma^{-1}( \Xi)$. Therefore we have $ \Lambda_U \subseteq \Xi^\uparrow$. Conversely, assume that $\z \in \Xi^\uparrow$. Then $\pi_\Sigma(\z) \in \Xi = \Lambda_U / \Sigma$ and so there exists some $\x \in \Lambda_U$ with 
     \[
     \pi_{\Sigma}(\x) = \pi_\Sigma(\z),
     \]
    i.e.~$\z - \x \in \mathrm{span}_\R(\Sigma)$. But now $\x, \z$ are integer vectors and $\Sigma$ is primitive, thus 
    \[
    \z - \x \in \mathrm{span}_\R(\Sigma) \cap \Z^{2m} = \Sigma \subseteq \Lambda_U.
    \]
    Therefore we get $\z = \x + (\z-\x) \in \Lambda_U$, which finishes our proof of the equality $\Lambda_U = \Xi^\uparrow$, and therefore also proves that the tuple $(\Sigma, (\w_i)_{i=1}^m)$ uniquely determines the lattice $\Lambda_U$.

    It remains to count the possible 
    $U \in \mathcal{A}_T(\boldsymbol{\theta})$ with $\Lambda_U = \Xi^\uparrow$. 
    To each $U$, we will associate, according to
    a fixed basis, a tuple of coefficient vectors $\mathbf{v}_1,\ldots,\mathbf{v}_q$. We will show that this tuple is isotropic and that each $\mathbf{v}_i$ is contained in a set of diameter $\ll T/\theta_i$. An application of Proposition~\ref{prop:translated_isotropic_frames} will then complete the proof.
    
    Let us choose a basis $\z_1, \ldots, \z_q$ of $\Sigma$ as in Lemma~\ref{lem:basis_successive_min_size}. Since $\Sigma$ is a primitive lattice in $\Lambda_U$, we may extend the basis $\z_1, \ldots, \z_q$ to a basis $\z_1,\ldots,\z_m$ of $\Lambda_U$.  Since $U \in \mathcal{A}_T(\boldsymbol{\theta})$, the successive minima $\sigma_1,\ldots,\sigma_q$ of $\Sigma_U = \Sigma$ satisfy 
    \[
    \theta_i \le \sigma_i \le 2 \theta_i \quad \text{ for $i =1,\ldots,q$;}
    \]
    by our choice of basis (Lemma~\ref{lem:basis_successive_min_size}), we thus have 
    \begin{equation}\label{eq:basis_of_sigma}
        |\z_i| \ll \theta_i \qquad \text{ for $i =1,\ldots,q$.}
    \end{equation}

    The lemma statement is trivially true if there are no admissible matrices, so we may assume that admissible matrices exist; we fix $U = (\x_1~\cdots~\x_m~\y_1~\cdots~\y_m)$ satisfying (1) and (2). 
 Then we may write 
    \begin{equation}\label{eq:bij}
         \x_i = \sum_{k=1}^m b_{i,k} \z_k 
    \end{equation}
    for some $b_{i,k} \in \Z$. 
    For $i=1,\ldots,q$, set $\mathbf{b}_i = (b_{1,i},\ldots,b_{m,i})$. 
    
    We show that the $\mathbf{b}_i$ lie in a box of diameter $\ll T/\theta_i$.
     Let us compare the first $m$ columns of $U$ with the first $m$ columns of another symplectic integer matrix $U_0$ with (1) and (2). We write
      \begin{equation*}
      \x_i^0 = \sum_{k=1}^m b_{i,k}^0 \z_k
    \end{equation*} and $\mathbf{b}_i^0 = (b_{1,i}^0,\ldots,b_{m,i}^0)$ for $i=1,\ldots,q$. 
      Since both $\x_i$ and $\x_i^0$ are contained in the preimage of $\w_i$ under $\pi_\Sigma$, we have \[
      \x_i - \x_i^0 = \mathbf{s} \in \mathrm{span}_{\R}(\Sigma)\cap \Z^{2m} = \Sigma,
      \]
      and since $\x_i$ and $\x_i^0$ lie inside the same column box, as defined by the constraints on $\mathcal{A}_T(\boldsymbol{\theta})$, we have $|\mathbf{s}| \ll T$. 
    For $i=1,\ldots,m$, write 
    \begin{equation*}
     \x_i - \x_i^0 = \sum_{j=1}^q a_{i,j} \z_j, \qquad a_{i,j} \in \Z.   
    \end{equation*}
    By our choice of basis (Lemma~\ref{lem:basis_successive_min_size}), setting
    $
    \mathbf{a}_i = (a_{1,i},\ldots,a_{m,i})$, we get $|\mathbf{a}_i | \ll T/\theta_i$, i.e. 
    \[
    |\mathbf{a}_i|= | \mathbf{b}_i - \mathbf{b}_i^0 | \ll T/\theta_i,
    \]
    which shows that $ \mathbf{b}_i$ lies in a box of diameter $\ll T/\theta_i$.

    Now, for our fixed admissible matrix $U = (\x_1~\cdots~\x_m~\y_1~\cdots~\y_m)$, we aim to bound the number of tuples $\widetilde{\y}_1,\ldots,\widetilde{\y}_m$ such that $\widetilde{U}=(\x_1~\cdots~\x_m ~\widetilde{\y}_1~\cdots~\widetilde{\y}_m)$ is also admissible, i.e.~such that $\widetilde{U}$ lies in $ \mathcal{A}_T(\boldsymbol{\theta})$ and satisfies (1) and (2).
    For this, we first show that there is a \emph{symmetric} matrix $M \in \mathcal{M}_m(\Z)$ such that 
    \[
    (\y_1 - \widetilde{\y}_1~\cdots~\y_m - \widetilde{\y}_m) = (\x_1~\cdots~\x_m)M,
    \]
    and that $\y_i - \widetilde{\y}_i \in \Sigma$ for all $1 \le i \le m$. 
    
    Since both $U$ and $\widetilde{U}$ are symplectic, we have 
    \[
    \omega(\x_i, \y_j - \widetilde{\y}_j) = 0
    \]
    for all $1 \le i,j \le m$.
    Thus the differences of the columns lie in the symplectic complement of the Lagrangian lattice spanned by the $\x_i$, and thus
    \[
        \y_j - \widetilde{\y}_j \in \Lambda_U^{\perp_\omega} \cap \Z^{2m} = \Lambda_U.
    \]
    Thus there exists some matrix $M =(m_{i,j})_{i,j=1}^m \in \mathcal{M}_m(\Z)$ with 
    \[
    (\y_1 - \widetilde{\y}_1~\cdots~\y_m - \widetilde{\y}_m) = (\x_1~\cdots~\x_m)M,
    \]
    that is,
    \[
    \y_i - \widetilde{\y}_i = \sum_{k=1}^m m_{k,i} \x_k \quad \text{ for $i=1,\ldots,m$}.
    \]
    Moreover, since $U$ is a symplectic matrix, we have 
    \[
    \omega (\x_i,\y_j) = 
    \begin{cases}
        1 & \text{ if $i =j$},
        \\
        0 & \text{ otherwise},
    \end{cases} \qquad 
     \omega (\y_j, \x_i) =  
    \begin{cases}
        -1 & \text{ if $i =j$},
        \\
        0 & \text{ otherwise},
    \end{cases}
    \]
    as well as $\omega(\y_i,\y_j) =0$ for all $1\le i,j \le m$; analogous relations hold for the columns of $\widetilde{U}$.
    Thus, by the bilinearity of $\omega$,  
    \[
    m_{j,i} = \omega\left(\sum_{k=1}^m m_{k,i} \x_k, \y_j\right) = \omega( \y_i - \widetilde{\y}_i, \y_j) =   \omega( \y_i, \y_j)  - \omega(\widetilde{\y}_i , \y_j)= -\omega(\widetilde{\y}_i , \y_j) .
    \]
   Moreover, $\y_i - \widetilde{\y}_i $ and $ \y_j - \widetilde{\y}_j$ lie in the isotropic lattice $\Lambda_U$, so 
   \[
   0 = \omega(\y_i - \widetilde{\y}_i, \y_j - \widetilde{\y}_j) = 
    \omega(\y_i , \y_j) 
    -
    \omega( \widetilde{\y}_i, \y_j ) 
    -
    \omega(\y_i ,  \widetilde{\y}_j)
    +
    \omega( \widetilde{\y}_i, \widetilde{\y}_j) = m_{j,i}-m_{i,j},
   \]
   so $M = M^t$ is a symmetric matrix.

    Since the columns of $U$ and $\widetilde{U}$ lie inside the same box, where the constraints are given by the definition of $\mathcal{A}_T(\boldsymbol{\theta}) \subseteq \mathfrak{B}_{\mathbf{K}}(T)$, we have $\|\y_i - \widetilde{\y}_i\|_2 \le 2\sqrt{2m}T$ for all $1 \le i \le m$.
    Thus we also have 
    \[
    | \pi_\Sigma( \y_i - \widetilde{\y}_i ) |
    \le 
    \| \pi_\Sigma( \y_i - \widetilde{\y}_i ) \|_2\le 2\sqrt{2m}T,
    \]
    and so, by Lemma~\ref{lem:nonzerolowerbound}, it follows that 
    \[
    \pi_\Sigma( \y_i - \widetilde{\y}_i ) = 0,
    \]
    i.e.~$ \y_i - \widetilde{\y}_i \in \mathrm{span}_{\R}(\Sigma) \cap \Z^{2m} = \Sigma$. Using the basis defined in \eqref{eq:basis_of_sigma}, we may write 
    \begin{equation}\label{eq:cij}
        \y_i - \widetilde{\y}_i = \sum_{j=1}^q c_{i,j} \z_j.
    \end{equation}
    Upon setting 
    \[
    \mathbf{c}_k = (c_{1,k},\ldots,c_{m,k})
    \]
    for $k=1,\ldots,q$, 
    it follows by our choice of basis that $|\mathbf{c}_k| \ll T/\theta_k$.

    By \eqref{eq:bij} and \eqref{eq:cij} we get 
    \[
    \mathbf{c}_j = \mathbf{b}_jM.
    \]
    Since $M$ is symmetric, we get (in row-vector notation)
    \[
    \langle \mathbf{c}_j, \mathbf{b}_i \rangle = 
    \mathbf{b}_jM \mathbf{b}_i^t = \mathbf{b}_jM^t   \mathbf{b}_i^t=
    \langle \mathbf{b}_j, \mathbf{c}_i \rangle,
     \]
     where $ \langle \cdot, \cdot \rangle $ is the usual Euclidean inner product.
     Now we define 
     $
     \mathbf{v}_i = (\mathbf{b}_i,\mathbf{c}_i)$,
     so that 
     \[
     \omega (\mathbf{v}_i, \mathbf{v}_j) =  \langle \mathbf{c}_j, \mathbf{b}_i \rangle-\langle \mathbf{c}_i, \mathbf{b}_j \rangle =0.
     \]
     This shows that $(\mathbf{v}_1,\ldots,\mathbf{v}_q)$ is a $q$-tuple of isotropic integer vectors in $\Z^{2m}$. Moreover, $\mathbf{b}_i$ lies in a set of diameter $ \ll T/\theta_i$ and $|\mathbf{c}_i|\ll T/\theta_i$. Therefore each vector $\mathbf{v}_i$ lies in a box of diameter $D_i \ll 1+ T/\theta_i \ll T/\theta_i$; here we used the fact that $\theta_q \le A^m T$ for fixed $A$.
     Thus, Proposition~\ref{prop:translated_isotropic_frames} gives that the number of possible $\mathbf{v}_1, \ldots, \mathbf{v}_q$ is 
     \begin{equation}\label{eq:boundonvi}
     \ll_\varepsilon T^\varepsilon \prod_{j=1}^q (T/\theta_j)^{2(m-j+1)} =
     T^{q(2m-q+1) +\varepsilon} \prod_{j=1}^q \theta_j^{-2(m-j+1)}.
     \end{equation}

The vectors $\mathbf b_1,\ldots,\mathbf b_q$ determine the first $m$ columns. Once these are fixed, we choose one possible completion; relative to this choice, the vectors $\mathbf c_1,\ldots,\mathbf c_q$ determine the final $m$ columns. Hence the map
\[
 U \mapsto (\mathbf v_1,\ldots,\mathbf v_q) 
\]
is injective, and the lemma statement follows from \eqref{eq:boundonvi}.   
\end{proof}

We are now ready to prove Proposition~\ref{prop:ATbound}.

\begin{proof}[Proof of Proposition~\ref{prop:ATbound}]
     By Lemma~\ref{lem:bounding_sigmaw}, the number of pairs $(\Sigma, (\w_i)_{i=1}^m)$ such that there exists some symplectic integer matrix $U =  \left(
    \x_1 ~ \cdots ~ \x_m ~ \y_1 ~ \cdots ~ \y_m
    \right) \in \mathcal{A}_T(\boldsymbol{\theta})$  with \begin{equation}\label{eq:cond}
    \Sigma = \Sigma_U \quad  \text{and} \quad \w_i = \pi_{\Sigma_U}(\x_i)
    \text{ for } i=1,\ldots,m
    \end{equation} can be bounded from above by
    \[
        \# \{ 
        (\Sigma, (\w_i)_{i=1}^m) \text{ satisfying \eqref{eq:cond}}
        \}
        \ll_\varepsilon T^{ (m-q)(m-q+1) +\varepsilon/2}  \prod_{i=1}^q \theta_i^{2(m-i+1)}.
    \]
    On the other hand, by Lemma~\ref{lem:boundingpersigmaw}, the number of admissible $U$ for each tuple $(\Sigma,(\w_i)_{i=1}^m)$ is 
    \[
    \ll_\varepsilon  T^{q(2m-q+1)+\varepsilon/2}
    \prod_{i=1}^q
    \theta_i^{-2(m-i+1)}. 
    \]
    Multiplying together both bounds yields the proposition.
\end{proof}

\subsection{Completing the proof of Theorem~\ref{thm:upperbound}} \label{section:completingproof}
We begin by bounding the number of matrices $U = (\x_1 ~ \cdots \x_m ~ \y_1~\cdots~\y_m) \in \Sp_m(\Z) \cap \mathfrak{B}_{\mathbf{K}}(T)$
such that the lattice $\Sigma_U$ has rank $q = 0$. By the definition of $\Sigma_U$ in \eqref{eq:def_SigmaU} and our lower bound on $A$ in \eqref{eq:choice_of_A}, this means that 
\begin{equation}\label{eq:lowerboudn}
    |\x| > 2 \sqrt{2m} T
\end{equation}
for all $\mathbf{0} \neq \x \in \Lambda_U$. Since each column $\x_i$ lies in a set of diameter $\ll T$, by Proposition~\ref{prop:translated_isotropic_frames} there are  
\[
    \ll_\varepsilon 
    T^\varepsilon \prod_{i=1}^m T^{2(m-i+1)}
    = T^{m(m+1)+\varepsilon}
\]
possible choices for $\x_1,\ldots,\x_m$. For each such tuple $\x_1,\ldots,\x_m$, there is at most one possibility for  $(\y_1,\ldots,\y_m)$. Indeed, for two possible completions 
$(\y_1,\ldots,\y_m)$ and
$(\widetilde{\y}_1,\ldots,\widetilde{\y}_m)$, we get 
\[
    \y_i - \widetilde{\y}_i \in \Lambda_U.
\]
But then, since $\y_i$ and $\widetilde{\y}_i$ lie in the same box of diameter $2\sqrt{2m}T$, we also have $| \y_i - \widetilde{\y}_i | \le 2\sqrt{2m}T$. By \eqref{eq:lowerboudn}, it follows that $\y_i = \widetilde{\y}_i$.

It remains to deal with matrices $U$ such that $\Sigma_U$ has rank $q \ge 1$.
Decompose the successive minima $\sigma_1 \le \cdots \le \sigma_q \le A^q T$ according to $\ll_\varepsilon T^{\varepsilon/2}$ dyadic ranges $\theta_i \le \sigma_i \le 2 \theta_i$. For each such range, there are 
\[
    \ll_\varepsilon T^{m(m+1)+\varepsilon/2}
\]
possible matrices $U$ by Proposition~\ref{prop:ATbound}. Summing over the dyadic ranges completes the proof. \qed

\section{The asymptotic formula for centred boxes}\label{section:asymptotic}

In this section, we specialise Maucourant's result from Section~\ref{section:prelim_lie} to the group of symplectic matrices.

\begin{proof}[Proof of Theorem~\ref{theorem:asymptotic}]
We now specialise the setting of Theorem~\ref{thm:mau} to the case of the connected, noncompact, semisimple real Lie group
$\Sp_m(\R)$ with finite centre, containing the irreducible lattice $\Sp_m(\Z)$. 

The associated Lie algebra is 
\[
    \mathfrak{sp}_m(\R) = \{ 
    X \in \mathcal{M}_{2m}(\R):~X^t\Omega + \Omega X = \mathbf{0}
    \},
\]
i.e.~the set of all matrices of the form 
\[
    \begin{pmatrix}
        A & B \\
        C & -A^t
    \end{pmatrix}, \quad \text{where} \quad A,B,C \in \mathcal{M}_m(\R), \quad B = B^t, \quad C= C^t.
\]
The Lie bracket is given via 
$
    [X,Y] = XY - YX.
$
We have the usual matrix exponential function $\exp$, and the Killing form is 
$
    B(X,Y) = 2(m+1)\operatorname{tr}(XY).
$
The
function
\[
    \theta:~  \mathfrak{sp}_m(\R) \to  \mathfrak{sp}_m(\R), \quad X \mapsto -X^t
\]
is the standard Cartan involution, 
yielding the Cartan decomposition 
$
    \mathfrak{sp}_m(\R) = \mathfrak{k}\oplus\mathfrak{p}, 
    $
    with
    \[
    \mathfrak{k} = 
    \{
    X \in  \mathfrak{sp}_m(\R):~X=-X^t 
    \}, \qquad 
    \mathfrak{p} = 
    \{
    X \in  \mathfrak{sp}_m(\R):~X=X^t 
    \}.
\]

A convenient maximal abelian subspace $\mathfrak{a} \subseteq \mathfrak{p}$ is given by 
\[
    \mathfrak{a} = 
    \{
    \operatorname{diag}(t_1,\ldots,t_m,-t_1,\ldots,-t_m):~t_1,\ldots,t_m \in \R
    \}.
\]
Indeed, it is clear that $\mathfrak{a}$ is abelian. To see that it is maximal,
let $X$ be a matrix  
\[
X=
\begin{pmatrix}
A&B\\
B&-A
\end{pmatrix}\in\mathfrak{p}
\]
that commutes with every 
\[
H_D=\begin{pmatrix}
D&\mathbf{0}\\
\mathbf{0}&-D
\end{pmatrix}
\in\mathfrak a, \quad \text{where} \quad D = \operatorname{diag}(t_1,\ldots,t_m).
\]
Then
\[
[H_D,X]
=
H_D X - X H_D 
=
\begin{pmatrix}
DA-AD&DB+BD\\
-(DB+BD)&DA-AD
\end{pmatrix}
\]
 vanishes for every diagonal matrix $H_D$, and so 
 \[
 0=(DA - AD)_{i,j} = (t_i - t_j)A_{i,j}.
 \]
 Since the $t_i$ may be chosen arbitrarily, it follows that $A_{i,j}=0$ for $i\neq j$. Similarly, we have 
 \[
    0 = (DB + BD)_{i,j} = (t_i + t_j)B_{i,j},
 \]
 which forces $B = \mathbf{0}$. 
Hence $X\in\mathfrak a$, proving that
$\mathfrak a$ is maximal abelian in $\mathfrak p$.

We now determine the restricted roots $\alpha \in \mathfrak{a}^\star$. We write $X \in \mathfrak{sp}_m(\R)$ as 
\[
X =
  \begin{pmatrix}
        A & B \\
        C & -A^t
    \end{pmatrix},  
\]
for some $A,B,C \in \mathcal{M}_m(\R), B = B^t, C= C^t$. Then we have 
\[
[H_D,X]
=
\begin{pmatrix}
\bigl((t_i-t_j)A_{i,j}\bigr)_{i,j}
&
\bigl((t_i+t_j)B_{i,j}\bigr)_{i,j}
\\[4pt]
-\bigl((t_i+t_j)C_{i,j}\bigr)_{i,j}
&
\bigl((t_i-t_j)A_{j,i}\bigr)_{i,j}
\end{pmatrix}.
\]
Defining $\varepsilon_i \in \mathfrak{a}^\star$ by  $\varepsilon_i (H_D) = t_i$, one may read off from the above that the restricted roots $\mathcal{R}$ in $\mathfrak{a}^\star$ are
\[
\mathcal{R} =
\{\pm ( \varepsilon_k - \varepsilon_\ell):  k < \ell\} \cup 
    \{\pm ( \varepsilon_k + \varepsilon_\ell):  k < \ell  \} \cup 
    \{\pm 2 \varepsilon_k\} .
\]
Moreover, each restricted root space is one-dimensional. Indeed, for each
restricted root the relations above leave exactly one free matrix parameter
(up to the symmetry conditions on $B$ and $C$); for instance, for   
$1\leq k<\ell\leq m$, one may compute that
\[
\mathfrak{g}_{\varepsilon_k-\varepsilon_\ell}
=
\left\{
\begin{pmatrix}
t\,(\mathbf{1}_{[i=k,j=\ell]})_{i,j} & 0\\
0 & -t\,(\mathbf{1}_{[j=k,i=\ell]})_{i,j}
\end{pmatrix}
:
t\in\mathbb{R}
\right\}.
\]
Hence we have
\[
m_\alpha=\dim(\mathfrak g_\alpha)=1
\qquad
\text{for all }\alpha\in\mathcal R.
\]
Choosing a regular element $\mathrm{diag}(t_1,\ldots,t_m,-t_1,\ldots,-t_m) \in \mathfrak{a}$ with $t_1 > \cdots > t_m >0$, one may determine the system of positive roots as 
\[
    \mathcal{R}^+ = \{ \varepsilon_k - \varepsilon_\ell: 1 \le k < \ell \le m  \} \cup 
    \{ \varepsilon_k + \varepsilon_\ell: 1 \le k \le \ell \le m  \}.
\]
We have 
\begin{align*}
    \beta = \sum_{\alpha \in \mathcal{R}^+} m_{\alpha}\alpha  
    &= \sum_{1 \le k < \ell \le m} (\varepsilon_k - \varepsilon_\ell) + \sum_{1 \le k \le \ell \le m} (\varepsilon_k + \varepsilon_\ell) \\
    &= \sum_{1 \le k < \ell \le m} \left( 
    (\varepsilon_k - \varepsilon_\ell) 
    +
    (\varepsilon_k + \varepsilon_\ell)
    \right) + 2\sum_{k=1}^m \varepsilon_k,
    \end{align*}
    where 
    \begin{align*}
    \sum_{1 \le k < \ell \le m} \left( 
    (\varepsilon_k - \varepsilon_\ell)
    +
     (\varepsilon_k + \varepsilon_\ell)
    \right)=
    2\sum_{k=1}^m 
    \left(
    \# \{\ell:~ k < \ell \le m
    \}
    \right)
    \varepsilon_k =
    2\sum_{k=1}^{m} 
    (m-k)
    \varepsilon_k,
\end{align*}
and therefore 
\[
    \beta = 2\sum_{k=1}^{m} 
    (m-k+1)
    \varepsilon_k.
\]

The standard representation of $\Sp_m(\R)$ is simply the inclusion 
\[
    \rho:~ \Sp_m(\R) \xhookrightarrow{} \GL_{2m}(\R).
\]
We have 
\[
    \exp(H_D) = \operatorname{diag} (e^{t_1},\ldots,e^{t_m}, e^{-t_1},\ldots,e^{-t_m}),
\]
and so the weights of the representation are $\pm \varepsilon_1, \ldots, \pm \varepsilon_m$. Therefore the weight polytope is 
\[
   \mathcal{C}_\rho = \operatorname{conv}(\pm \varepsilon_1,\ldots,\pm\varepsilon_m) = \left\{
    \sum_{i=1}^m \mu_i \varepsilon_i:~ 
    \sum_{i=1}^m |\mu_i| \le 1
    \right\}.
\]
For $r>0$, we have $r^{-1} \beta \in \mathcal{C}_\rho$ if and only if
\[
    r^{-1} 2 \sum_{k=1}^m (m-k+1)  = r^{-1}m(m+1)
 \le 1.
\]
Therefore $d = \min \{r:~ r \text{ admissible} \} = m(m+1)$. 
We have 
\[
    d^{-1} \beta = \sum_{k=1}^m \frac{2(m-k+1)}{m(m+1)}\varepsilon_k;
\]
so $d^{-1}\beta$ lies in the relative interior of the $(m-1)$-dimensional face
\[
F_\beta = \operatorname{conv} \{\varepsilon_1,\ldots,\varepsilon_m\} = \varepsilon_1 + \operatorname{conv} \{\mathbf{0}, \varepsilon_2-\varepsilon_1,\ldots,\varepsilon_m-\varepsilon_1\}. 
\]
Since $\mathcal C_\rho$ has dimension $m$, we get
$
e = \operatorname{codim}(F_\beta) -1 =0. 
$

Finally, applying Theorem~\ref{thm:mau} to $\Sp_m(\Z) \subset \Sp_m(\R)$ with the norm $\|\cdot \| = |\cdot|$ gives
\[
N_m(\mathfrak{B}_{\mathbf{0}}(T)) =
  \# \{ 
  g \in \Sp_m(\Z):~|g|\le T
  \} \sim a_m T^{d} (\log(T))^e = a_m T^{m(m+1)}
\]
as $T \to \infty$.
\end{proof}

\end{document}